\documentclass[10pt]{article} 
\usepackage{latexsym,amsfonts,amssymb,amsmath} 
\usepackage{graphicx}
\newtheorem{thm}{Theorem} 
\newtheorem{crl}{Corollary} 
\newtheorem{prp}{Proposition} 
\newtheorem{lm}{Lemma}

\newtheorem{dfn}{Definition} 
\newenvironment{proof}{\noindent{\bf Proof.}}{\noindent$\Box$\par\medskip}

\newcommand{\suba}{\mathrm {subalg}}

\newcommand{\rad}{\mathrm {rad}} 
\newcommand{\spa}{\mathrm {span}}
\newcommand{\ad}{\mathrm {ad}}
\newcommand{\bideg}{\mathrm {bideg}}
\newcommand{\F}{\mathcal{F}}
\newcommand{\g}{\mathfrak{g}}
\newcommand{\be}{\bar\eta}

\newcommand{\fe}{\varphi}  

\newcommand{\sudda}[1]{}

\begin{document} 
\date{\today} 
\title{A Gorenstein numerical semi-group ring having a transcendental series of Betti numbers }   
\author {Clas L\"ofwall, Samuel Lundqvist, and Jan-Erik Roos}
\maketitle

\section{Introduction}

Around 1969 Oscar Zariski asked Ernst Kunz whether there were any
relations between Gorenstein rings and symmetric numerical semigroup
rings (cf. the introduction in \cite{Kunz}, where this question was settled by Kunz ; more details
below).
Presumably Zariski was inspired both by the thesis of J\"urgen Herzog \cite{Her} and by his own  results that to {\it plane
curve singularitites} there is always associated a numerical
semigroup ring which was a {\it complete intersection} (i.e. a
special case of a Gorenstein ring) \cite{Zar}. According to Tate \cite{Ta} a complete
intersection always has a rational Poincar\'e-Betti series.
Here we will prove that there are Gorenstein numerical semigroup
rings whose Poincar\'e-Betti series is transcendental.
More precisely, here is the smallest example we have been able to find:
let $k$ be a field of characteristic zero and $R$ the subring of the polynomial ring $k[t]$,
generated by the twelve monomials
$$
t^{36},t^{48},t^{50},t^{52},t^{56},t^{60},t^{66},t^{67},t^{107},t^{121},t^{129},t^{135}.
$$ 
Then $R$ is a Gorenstein ring and the generating series (the Poincar\'e-Betti series of $R$)
$$
P_R(z)=\sum_{i\geq 0} |{\rm Tor}_i^R(k,k)|z^i
$$
is rationally related to the infinite product
$$
\prod_{n=1}^{\infty}{(1+z^{2n-1})^2\over(1-z^{2n})^2}
$$
and thus transcendental. If we divide $R$ by the non-zero divisor in the
maximal ideal 
corresponding to e.g. $t^{36}$ we obtain an {\it artinian} Gorenstein ring of
embedding dimension 11 having an irrational series of Betti numbers and
this might be the smallest possible example, cf. B\o gvad \cite{Bo} and \cite[pages 459--461]{Roos1}. Cf. also  \cite{Roos2} for the
skew-commutative case.
For even more precise results and details, cf. Theorem \ref{thm1} below.
But even more interesting are the methods of proofs related to Golod homomorphisms,
so-called large maps, graded Lie algebras etc.
and thereby to results by Gulliksen, Levin, Avramov, L\"ofwall, B{\o}gvad and others. In particular we wish to draw the attention to useful decomposition
results for finitely presented graded Lie algebras in section \ref{sec:proofs} (theorems \ref{modul}, \ref{one} and \ref{iso}).

\section{Symmetric numerical semigroups and Gorenstein rings}

A numerical semigroup $S$ is an additive semigroup of natural numbers generated by a set of
numbers $0<g_1<g_2< \ldots <g_n$ such that $\gcd(g_1,g_2,\ldots,g_n)=1$.
We write $S=(g_1,\ldots,g_n)$ 
 Let $k$ be a field of characteristic zero.
The numerical semigroup ring $k[S]$ of $S$ over $k$ is by definition the subring of the polynomial
ring $k[t]$ generated by the monomials 
$$
t^{g_1},t^{g_2},\ldots,t^{g_n}
$$
The Poincar\'e-Betti series of $k[S]$ is by definition the generating series
$$
P_{k[S]}(z) = \sum_{i\geq 0} |{\rm Tor}^{k[S]}_{i}(k,k)| z^i
$$	
(for a $k$-vector space $V$ we denote by $|V|$ the dimension of $V$ over $k$).
In \cite{Fro-Ro} Ralf Fr\"oberg and the third author found by modifying an earlier result by Roos-Sturmfels \cite{Ro-St} that
the following semigroup
$$
S=(18,24,25,26,28,30,33)
$$
and the corresponding $k[S]$ had a series $P_{k[S]}(z)$ that was an explicit transcendental function.
In the present paper we will prove that there are semigroups $S$ such that $G=k[S]$  is a
{\it Gorenstein} ring and such that $P_G(z)$ is a transcendental function.

\bigskip

Let us recall that ${\bf N}\setminus S$ is finite, and let $F(S)$ the largest integer that is not in
$S$ (it is called the Frobenius number of $S$).
Now $S$ is called symmetric if for every $n \in {\bf Z}$ either $n\in S$ or $F(S)-n \in S$.
The following result was proved in the paper \cite{Kunz} (mentioned in the introduction): $k[S]$ is a Gorenstein ring if and only if $S$ is symmetric. 
Now in \cite{Ro-G-San} there are described infinitely many different ways to associate a symmetric semigroup to a given semigroup and
we will see that from {\it a homological point of view} these different ways give essentially the same result.
Let us start with any numerical semigroup $S$ and let us recall the recipe from \cite{Ro-G-San}.
Let $F(S)$ be the Frobenius number of $S$ just defined 
($F(S)+1$ is also called the conductor of $S$). Let us also denote by $PF(S)$ the set of pseudo-Frobenius
numbers of $S$, i.e. those integers $z \notin S$ such that $z+ S\setminus \{0\}\subseteq S$. 
 The cardinality of $PF(S)$ is also called the type of $S$.
Now suppose that $S=(g_1,\ldots,g_n)$ and that $PF(S)=(n_1,\ldots,n_t)$ and let $\bar g$ be an {\it odd} integer
such that $\bar g \geq 3F(S)+1$. Then
$$
\bar S_{\bar g} = (2g_1,2g_2,\ldots,2g_n,\bar g -2n_1,\bar g -2n_2,\dots,\bar g -2n_t)
$$
is a symmetric numerical semigroup such that $S = \{n | 2n \in {\bar S}_{\bar g}\}$.
Note that this gives infinitely many $\bar S_{\bar g}$ but they are all {\it essentially}
the same from a homological point of view. We will only illustrate this with the special case
$S = (18,24,25,26,28,30,33)$ taken from \cite{Fro-Ro} and \cite{Ro-St}. Now in this case $F(S)=65$ and $PF(S)=(65,45,38,34,31)$.
Thus the odd integers $\geq 3F(S)+1$ are the integers $197,199,201,\ldots$ and
$$ {\bar S}_{197} =(36,48,50,52,56,60,66,67,107,121,129,135)$$
and
$$ {\bar S}_{199} =(36,48,50,52,56,60,66,69,109,123,131,137)$$
etc.
We start by analyzing the smallest case ${\bar S}_{197}$.
Let us denote the corresponding numerical semigroup ring by
$R_{197}$ which is a Gorenstein ring of dimension 1 and a domain, since it is a subring of $k[t]$.

We therefore have to determine the series
\begin{equation} \label{eq2}
P_{R_{197}}(x,y) =\sum_{i,j}|{\rm Tor}^{R_{197}}_{i,j}(k,k)|x^iy^j 
\end{equation}
where $j$ refers to the grading of $R_{197}$.
For this analysis we use Macaulay2 \cite{M2} (working over ${\bf Q}$). 
Recall that the ring $R_{197}$ can be obtained by
introducing a ring $R={\bf Q}[t]$, another ring $T={\bf Q}[a,b,c,d,e,f,g,h,i,j,k,l]$ and a map $\phi$ between them 
 which is defined by $a\rightarrow t^{36}, b\rightarrow t^{48}$ etc.
 The kernel of $\phi$ is an ideal $J$ in $T$ and
$$
{\bf Q}[t^{36},t^{48},t^{50},t^{52},t^{56},t^{60},t^{66},t^{67},t^{107},t^{121},t^{129},t^{135}] \cong  T/J \cong R_{197}
$$
It is important for us to keep track of the gradings and therefore the Macaulay2 code is:
\bigskip
{\tt

R:=QQ[t]

T:=QQ[a..l,Degrees=> $\{\{36\},\{48\},\{50\},\{52\},\{56\},\{60\},\{66\},\{67\},\{107\},\{121\},\{129\},\{135\}\}$]

phi=map(R,T,$\{ t\hat{\hbox{\ }}36,t\hat{\hbox{\ }}48,t\hat{\hbox{\ }}50,
t\hat{\hbox{\ }}52,t\hat{\hbox{\ }}56,t\hat{\hbox{\ }}60,
t\hat{\hbox{\ }}66,t\hat{\hbox{\ }}67,
t\hat{\hbox{\ }}107,
t\hat{\hbox{\ }}121,
t\hat{\hbox{\ }}129,
t\hat{\hbox{\ }}135\}$)

ker phi

J=trim(oo)
}

\bigskip
The ideal $J$ obtained is minimally generated by the following elements:
$$
b^2-af,c^2-bd,cd-ag,d^2-be,de-bf,a^3-bf,e^2-df,ef-cg,a^2b-f^2,a^2c-eg,a^2f-g^2,
abc-h^2,adh-bi,
$$
$$
ci-aj,aeh-di,afh-ei,bch-ak,bdh-fi,agh-bj,cj-al,beh-al,gi-dj,
ceh - dj, bfg - hi,ej - bk,
$$
$$
cfh -bk,a^2i-ck,dfh-ck,fj-dk,bgh-dk,cgh-bl,ek-cl,
dgh-cl,gj-dl,f^2h-dl,eg^2-hj,egh-fk,abi-el,
$$
$$
fgh-a^2j,gk-fl,adi-fl,abdf-hk,a^2k-gl,
abdg-hl, adfg-i^2,afg^2-ij,bhj-ik,j^2-il,bfh^2-il,
$$
\begin{equation} \label{eq3}
jk-bhl,fhk-jl,k^2-ehl,aij-kl,ei^2-l^2
\end{equation}
As we said above the ring $R_{197}=T/J$ is a Gorenstein ring of Krull dimension 1 and it is a subring
of ${\bf Q}[t]$, thus a domain. We now divide out by the non-zero divisor $a$
of degree 36. Now $(T/J)/(a)$ is an artinian Gorenstein ring which also can be quickly
determined by Macaulay2 as follows:
We start by simplifying the ideal $J$ by putting $a=0$
\medskip

{\tt substitute(J,$\{$a=>0$\}$)

I=trim(oo)
}
\medskip

We get a new graded ideal $I$ which is easier, since it is minimally generated by:
$$
b^2,c^2-bd,cd,d^2-be,bf,de,e^2-df,ef-cg,f^2,eg,g^2,h^2,bi,ci,di,ei,bch,bdh-fi,bj,cj,beh,gi-dj,
$$
$$
ceh-dj,hi,ej-bk,cfh-bk,ck,dfh,fj-dk,bgh-dk,cgh-bl,ek-cl,dgh-cl,dl,gj,hj,fk,el,
$$
$$
fgh,fl,gk,hk,gl,hl,i^2,ij,ik,il,j^2,jk,jl,k^2,kl,l^2
$$
and it can be considered in the ring ${\bf Q}[b,c,d,e,f,g,h,i,j,k,l]$, so that
$R_{197}/(a) \cong  {\bf Q}[b,c,d,e,f,g,h,i,j,k,l]/I$
and furthermore
\begin{equation} \label{eq4}
P_{R_{197}/(a)}(z,1) = P_{R_{197}}(z,1)/(1+z)
\end{equation}
since $a$ is a nonzero divisor.
 Note that in (\ref{eq4}) we have taken the total degree, but  everything is still
also graded as follows:
$$
b,c,d,e,f,g,h,i,j,k,l\, {\rm\, have\, the\, degrees\,} 48,50,52,56,60,66,67,107,121,129,135
$$
but it is
unwieldly to work with such high degrees. We therefore check the possible gradings of the last ideal
$I$. We therefore temporarily denote the possible degrees of the variables $b,c,d,e,f,g,h,i,j,k,l$
by the same letters, and in order to find all those integral degrees of the variables for which the relations in $I$ are still
homogeneous, we have to solve the following linear equations for integer solutions corresponding the non-monomial (i.e. the
binomial) relations in $I$, where the relation $c^2-bd$ gives the linear equation $2c-b-d=0$ etc.:

\begin{equation*}
\begin{aligned}[c]
2c-b-d& = 0\cr
2d-b-e& = 0\cr
2e-d-f& = 0\cr
e+f-c-g& = 0\cr
b+d+h-f-i& = 0\cr
\end{aligned}
\begin{aligned}[c]
g+i-d-j& = 0\cr
c+e+h-d-j& = 0 \cr
e+j-b-k& = 0\cr
c+f+h-b-k& = 0\cr
f+j-d-k& = 0\cr
\end{aligned}
\begin{aligned}[c]
b+g+h-d-k& = 0\cr
c+g+h-b-l& = 0\cr
e+k-c-l& = 0\cr
d+g+h-c-l& = 0\cr
\end{aligned}
\end{equation*}
We have 14 equations for the 11 unknowns $b,c,d,e,f,g,h,i,j,k,l$ and the result is:
(we have three constants $c_1,c_2,c_3$):

\begin{equation} \label{eq5}
\begin{aligned}[c]
b& = c_1\cr
c& = (c_1+c_2)/2\cr
d& = c_2\cr
e& = 2c_2-c_1\cr
f& = 3c_2-2c_1\cr
g& = (9c_2-7c_1)/2\cr
\end{aligned}
\begin{aligned}[c]
h& = c_3\cr
i& = c_3-2c_2+3c_1\cr
j& = (2c_3+3c_2-c_1)/2\cr
k& = (2c_3+7c_2-5c_1)/2\cr
l& = c_3+5c_2-4c_1\cr
\end{aligned}
\end{equation}

Thus $c_1=1$ and $c_2=1$ give the solutions
$b=c=d=e=f=g=1$ and $h=c_3,\quad i=j=k=l=c_3+1$
so that the minimal choice for a positive integral grading is $c_3=1$ which gives the final minimal integral grading
$b=c=d=e=f=g=h=1, \quad i=j=k=l=2$.
Now we transform our previous ideal $I$ to a ring with this last grading:
\bigskip

{\tt
JE:=QQ[b..l,Degrees => $\{\{1\},\{1\},\{1\},\{1\},\{1\},\{1\},\{1\},\{2\},\{2\},\{2\},\{2\}\}$]

I2 = substitute(I,JE)
}

\medskip

\noindent Thus we now have a ring ${\overline R}_{197} \cong JE/I2$ which we will study in detail.

\noindent The maximal ideal $(b,c,d,e,f,g,h,i,j,k,l)$ is generated by $b,c,d,e,f,g,h$ of degree $1$
and $i,j,k,l$ of degree 2. The square of the maximal ideal is generated by
$gh, fh, eh, dh, ch, bh, fg, dg, cg, bg, df, cf, ce,be, bd, bc$ and
the cube of the maximal ideal is generated by $cl, bl, dk, bk, dj, fi, bdg$
and finally the fourth power of the maximal is generated by $bcl$ which is also the socle
of ${\overline R}_{197}$. This and the relations $I2$ show that ${\overline R}_{197} \cong JE/I2$ is the trivial
 extension of 
 the ring 
\begin{equation} \label{eq6}
{\cal S}={\bf Q}[b,c,d,e,f,g]/(b^2,c^2-bd,cd,d^2-be,de,bf,e^2-df,ef-cg,eg,f^2,g^2)
\end{equation}
with an ${\cal S}$-module $M$ that is generated by $h,i,j,k,l$ in $\bar R_{197}\cong JE/I2$.
Thus ${\overline R}_{197} = {\cal S}\propto M$ and
 $M$ can be defined as the cokernel of the map
 $$
{\cal S}^{28} \longrightarrow {\cal S}^5 
$$
defined by the $5 \times 28$ matrix over ${\cal S}$:
 (the generators $h,i,j,k,l$ correspond
to the 5 rows of this matrix):

\bigskip

\resizebox{\linewidth}{!}{
$\begin{pmatrix} 
0&0&0&0&0&0&0&0&0&0&0&0&0&0&0&0&0&0&fg&dg&cg&bg&df&cf&ce&be&bd&bc \\
0&0&0&0&0&0&0&0&0&0&0&0&0&g&e&d&c&b&0&0&0&0&0&0&0&0&-f&0 \\
 0&0&0&0&0&0&0&0&g&f&e&c&b&-d&0&0&0&0&0&0&0&0&0&0&-d&0&0&0& \\
0&0&0&0&g&f&e&c&0&-d&-b&0&0&0&0&0&0&0&0&0&0&-d&0&-b&0&0&0&0 \\
g&f&e&d&0&0&-c&0&0&0&0&0&0&0&0&0&0&0&0&-c&-b&0&0&0&0&0&0&0
\end{pmatrix}$
}

\bigskip

Example: left matrix multiplication of the row matrix $h,i,j,k,l$ in $JE$ with the seventh column of the matrix above
gives the relation $ke-cl$ in $JE$.

Next we use the result proved by Avramov and Levin \cite{Av-Le} that since ${\overline R}_{197}$ is Gorenstein
the natural map ${\overline R}_{197} \longrightarrow {\overline R}_{197}/(bcl) $ is a Golod map. Let us put ${\overline R}={\overline R}_{197}/(bcl)$
With this notation the Golod condition implies that
$$
P_{\overline R}(z)= {P_{{\overline R}_{197}} \over 1-z^2 P_{{\overline R}_{197}}(z)}
$$
which we will write in the form
\begin{equation} \label{eq7}
{{1\over P_{{\overline R}_{197}}(z)}= {1\over P_{\overline R}(z)}+z^2}
\end{equation}
Now $bcl$ lies in $M$ and therefore ${\overline R}$ is 
 a new trivial extension ${\cal S} \propto {\overline M}$
where ${\overline M} = M/(bcl)$. 
Now, according to a result
 of Gulliksen \cite{Gu}
\begin{equation} \label{eq8}
P_{\overline R}(z) = {P_{\cal S}(z) \over 1-zP_{\cal S}^{\overline M}(z)}
\end{equation}
In order to determine  ${\rm Ext}^*_{\cal S}({\overline M},k)$
we first observe that  among the generators of ${\overline M}$, $h$ plays
 a special role since it has degree $1$, while the other generators $i,j,k,l$ have degree $2$. Let $N$ be the submodule of ${\overline M}$ generated by $h$.
We have an exact sequence of ${\cal S}$-modules:
\begin{equation} \label{eq9}
  0 \longrightarrow N \longrightarrow {\overline M}  \longrightarrow {\overline M}/N \longrightarrow 0 
\end{equation}
which gives a long exact sequence of ${\rm Ext}_{\cal S}^*(k,k)$-modules:
$$ \ldots  \rightarrow {\rm Ext}_{\cal S}^{n}({\overline M}/N,k) \rightarrow  {\rm Ext}_{\cal S}^{n}({\overline M},k)\rightarrow
  {\rm Ext}_{\cal S}^{n}(N,k) \rightarrow
{\rm Ext}_{\cal S}^{n+1}({\overline M}/N,k) \rightarrow {\rm Ext}_{\cal S}^{n+1}({\overline M},k)\rightarrow \ldots
$$
We now claim that 
\begin{equation} \label{eq10}
{\rm Ext}_{\cal S}^*({\overline M},k)\longrightarrow {\rm Ext}_{\cal S}^*(N,k)
\end{equation}
is an epimorphism. But (\ref{eq10}) is a map of (left) ${\rm Ext}_{\cal S}^*(k,k)$-modules.
The important thing now is that (\ref{eq10}) is an epimorphism in degrees $*=0$ and $*=1$
(direct calculation). Therefore if we prove that ${\rm Ext}_{\cal S}^*(N,k)$ is generated as
a (left) ${\rm Ext}_{\cal S}^*(k,k)$-module by its elements in degree 0 and 1 it will follow
that (\ref{eq10}) is an epimorphism in all degrees.
But the annihilator of $(h)$ in ${\overline R} = {\cal S}\propto {\overline M}$ is
$(fg,df,be,bc,bdg,h,i,j,k,l)$. Thus we have an exact sequence of ${\cal S}$-modules:
\begin{equation} \label{eq11}
0 \longrightarrow (fg,df,be,bc,bdg) \longrightarrow {\cal S} \longrightarrow N \longrightarrow 0
\end{equation}
But the five generating elements in the ideal to the left in (\ref{eq11}) are all annihilated by the maximal ideal in ${\cal S}$.
Thus if we take ${\rm Ext}_{\cal S}^*(.,k)$ of the sequence (\ref{eq11}) it follows that ${\rm Ext}_{\cal S}^*(N,k)$
is indeed generated as an ${\rm Ext}_{\cal S}^*(k,k)$-module by its elements of degree $0$ and $1$
 Thus we have indeed a short exact sequence 
\begin{equation} \label{eq12} 
0\longrightarrow {\rm Ext}_{\cal S}^{*}({\overline M}/N,k)\longrightarrow {\rm Ext}_{\cal S}^{*}({\overline M},k)
\longrightarrow {\rm Ext}_{\cal S}^{*}(N,k)\longrightarrow 0 
\end{equation}
so that in particular 
\begin{equation} \label{eq13}
 P_{\cal S}^{\overline M}(z) = P_{\cal S}^{{\overline M}/N}(z)+P_{\cal S}^N(z)
\end{equation}
Now the exact sequence (\ref{eq11}) shows that
$P_{\cal S}^N(z)=1+5zP_{\cal S}(z)$ and ${\overline M}/N$ is generated by
the $4$ elements $i,j,k,l$ which are annihilated by the maximal ideal of
${\cal S}$. Therefore $P_{\cal S}^{{\overline M}/N}(z)=4P_{\cal S}(z)$ so that
 $P_{\cal S}^{\overline M}(z) = 1+(4+5z)P_{\cal S}(z)$
and finally using (\ref{eq8})
\begin{equation} \label{eq14}
P_{\overline R}(z) ={P_{\cal S}(z)\over 1-z(1+(4+5z)P_{\cal S}(z))} 
\end{equation}
Now, if we combine (\ref{eq14}), rewritten in the form
$$
{1\over P_{\overline R}(z)} = {1-z\over P_{\cal S}(z)}-4z-5z^2
$$  
with (\ref{eq7}) we obtain at last:
$$
{1\over P_{{\overline R}_{197}}(z)}={1-z\over P_{\cal S}(z)}-4z-4z^2
$$
whose more precise graded form is deduced as follows:

1) Replace the formula (\ref{eq7}) by
\begin{equation} \label{eq15}
{{1\over P_{{\overline R}_{197}}(x,y)}= {1\over P_{\overline R}(x,y)}+x^2y^4}
\end{equation}

2) Replace the formula (\ref{eq8}) by
\begin{equation} \label{eq16}
P_{\overline R}(x,y) = {P_{\cal S}(x,y) \over 1-xyP_{\cal S}^{\overline M}(x,y)}
\end{equation}
3) Replace the formula (\ref{eq13}) by
\begin{equation} \label{eq17}
 P_{\cal S}^{\overline M}(x,y) = P_{\cal S}^{{\overline M}/N}(x,y)+P_{\cal S}^N(x,y)
\end{equation}
where $P_{\cal S}^{{\overline M}/N}(x,y)=4y P_{\cal S}(x,y)$ and
$P_{\cal S}^N(x,y)=(4xy^2+xy^3)P_{\cal S}(x,y)$
so that 
$$
 P_{\cal S}^{\overline M}(x,y) =4y P_{\cal S}(x,y)+4(xy^2+xy^3)P_{\cal S}(x,y)
$$
and therefore finally
$$
{1\over P_{\overline R}(x,y)}={1-xy \over P_{\cal S}(x,y)}-4xy^2-4x^2y^3-x^2y^4. 
$$

\begin{thm} \label{thm1}
Let $R_{197}$ be the numerical semigroup ring generated by
$$
t^{36},t^{48},t^{50},t^{52},t^{56},t^{60},t^{66},t^{67},t^{107},t^{121},t^{129},t^{135}
$$
as a subring of $k[t]$ ($k$ a field of characteristic 0) and let ${\overline R}_{197}$ be $R_{197}$ divided by the
non-zero divisor $t^{36}$. Both these rings are Gorenstein rings and the bigraded 
Poincar\'e-Betti series for ${\overline R}_{197}$ where the first $7$ of the remaining 11 generators are
given the degree $1$ and the last $4$ generators are given the degree $2$ is given by the formula
$$
{1\over P_{{\overline R}_{197}}(x,y)}={1-xy\over P_{\cal S}(x,y)}-4xy^2-4x^2y^3
$$
where $\cal S$ is the ring given in (\ref{eq6}) above and
$$
1/P_{{\cal S}}(x,y) = (1+1/x)/{\cal S}^!(xy)-{\cal S}(-xy)/x
$$
where ${\cal S}(t)=1+6t+10t^2+t^3$ is the Hilbert series of ${\cal S}$
and
$$
{\cal S}^!(t) = {1\over (1+t)(1-2t)^2(1-3t+t^2)}\prod_{n=2}^{\infty}{(1+t^{2n-1})^2\over(1-t^{2n})^2}
$$ 
is the Hilbert series of the Koszul dual of ${\cal S}$. 
\end{thm}
\begin{proof}
Everything follows from the computations above, except the formula for $S^!(t)$ which will be analyzed and proved in sections \ref{sec:irr} and \ref{sec:proofs}.
\end{proof}


\medskip

\medskip
Remark 1.-- If we put $y=1$ in all the formulae in Theorem \ref{thm1}, we get the ordinary Poincar\'e-Betti
series of Betti numbers.
\medskip

Remark 2.-- It follows from what we have proved above that the
 natural map ${\overline R} \longrightarrow {\overline R}/(h)$
is {\it large} in the sense of Levin \cite{Le1}, i.e. the natural map
\begin{equation}  \label{eq18}
 {\rm Ext}^*_{\overline R/(h)}(k,k) \longrightarrow {\rm Ext}^*_{\overline R}(k,k)  
\end{equation}
is a monomorphism.
\medskip
(In his lecture notes about Golod homomorphisms \cite{Le2}, Levin calls such maps
``co-Golod'' maps.)
\medskip
Indeed, the map ${\overline R} \longrightarrow {\overline R}/(h)$
is the same as the natural map (here $N = (h)$ in $\overline M$)
\begin{equation} \label{eq19}
{\cal S}\propto {\overline M} \longrightarrow {\cal S}\propto {\overline M}/N 
\end{equation}
Furthermore it is known that the ext-algebra of any trivial extension
$C\propto L$ (where $L$ is a $C$-module sits in the middle of a Hopf
algebra extension:
$$
k \longrightarrow T(s^{-1}{\rm Ext}^*_C(L,k)) \longrightarrow {\rm Ext}^*_{C\propto L}(k,k) \longrightarrow {\rm Ext}^*_C(k,k) \longrightarrow k   
$$
where $T$ is the graded tensor algebra and $s^{-1}$ means that we push the degrees upwards one step.
Thus if we take the extalgebras in (\ref{eq19}) we obtain that (\ref{eq18}) is a monomorphism if we can prove that
\begin{equation}  \label{eq20} 
{\rm Ext}^*_{\cal S}({\overline M}/N,k) \longrightarrow {\rm Ext}^*_{\cal S}({\overline M},k)
\end{equation}
is a monomorphism. But this is a consequence of (\ref{eq12}) above.

\section{Irrationality of the Poincar\'e-Betti series of $\cal S$.} \label{sec:irr}
Let  $\cal S$ be the ring
\begin{equation} \label{eq21}
{\cal S} = {{\bf Q}[b,c,d,e,f,g] \over (b^2,c^2-bd,cd,d^2-be,e^2-df,de,bf,ef-cg,eg,f^2,g^2)} 
\end{equation}
We have already published in \cite{Fro-Ro} and \cite{Ro-St} the result that
\begin{equation} \label{eq22}
1/P_{\cal S}(z) =(1+1/z)/{\cal S}^!(z)-(1-6z+10z^2-z^3)/z 
\end{equation}
where
$$
{\cal S}^!(z)= {1\over (1+z)(1-2z)^2(1-3z+z^2)}\prod_{n=2}^{\infty}{(1+z^{2n-1})^2\over(1-z^{2n})^2}
$$
The proofs there are however incomplete.
In this section, we will briefly indicate a general way of obtaining
this result as a part of a general theory and this is completed by
some general theorems in the next section. When it comes to analyzing the
Hilbert series of (24) below there are several technical details
only alluded at in this section and they will be completely treated
in section 4 below. But the sketchy treatment of the Hilbert series of
(24) given here can be considered as
a motivation and background for the hard work in section 4.

First we observe that the Hilbert series of ${\cal S}$ is
$1+6t+10t^2+t^3.$
Furthermore the third power of the maximal ideal $m=(b,c,d,e,f,g)$ of ${\cal S}$ is generated by $bdg$.
Thus $m^3=(b)m^2$ and since $b^2=0$ we can use another theorem of Levin, namely 
\cite[ Theorem 2.12, page 33]{Le2} for $n=3$ which says that
$$
{\cal S} \longrightarrow {\cal S}/(bdg)
$$
is a Golod map and that 
\begin{equation} \label{eq23}
P_{{\cal S}/m^3}(z) = P_{{\cal S}}(z)/(1-z^2 P_{{\cal S}}(z))  
\end{equation}
Therefore it is sufficient to concentrate our homological efforts on the ring
$(T,n) = {\cal S}/m^3$ for which the cube of the maximal ideal $n$ is 0.
But now we can use the result of Clas L\"ofwall \cite{Lo2} which says
that
\begin{equation} \label{eq24}
1/P_T(z) = (1+1/z)/T^!(z) - T(-z)/z  
\end{equation}
where $T(-z)=1-6z+10z^2$ is the Hilbert series of $T$ at $-z$ and  $T^!(z)$
is the Hilbert series of the Koszul dual $T^!$ of $T$.
It therefore remains to determine this last Hilbert series.
Recall that it follows from (\ref{eq21}) and \cite{Lo2} that
\begin{equation} \label{eq25}
T^! ={k<B,C,D,E,F,G> \over([B,C],C^2+[B,D],D^2+[B,E],[C,E],[C,F],E^2+[D,F],[B,G],[E,F]+[C,G],[D,G],[F,G])}
\end{equation}
where  $k<B,C,D,E,F,G>$ is the free associative algebra in the variables $B,C,D,E,F,G$ of degree 1 
which are dual to
$b,c,d,e,f,g$ and where $[,]$ denotes the graded commutator, so that e.g. $[B,C]=BC+CB$ etc.
Now we observe that $T^!$ is the enveloping algebra of the graded
Lie algebra  $\eta = \eta_T$ which is the quotient
of the free graded Lie algebra on the generators $B,C,D,E,F,G$ of degree $1$ by 
the Lie ideal generated by the relations in (\ref{eq25}) .
Thus if
 $$
\eta = \eta_1 \oplus \eta_2 \oplus \ldots \oplus \eta_n \oplus \ldots
$$
the Poincar\'e-Birkhoff-Witt theorem tells us that
$$
T^!(z) = \prod_{n \geq 1}{(1+z^{2n-1})^{\eta_{2n-1}} \over (1-z^{2n})^{\eta_{2n}}}
$$
where $\eta_i$ in the exponents should be interpreted as ranks. We now turn to the problem of determining the $\eta_i$.
\medskip

For this an essential role is played by the program {\tt liedim} written by Clas L\"ofwall \cite{Lo1}.
Recall that this program (which runs under Mathematica and can be downloaded from 

\noindent {\tt http://www2.math.su.se/\~{}clas/liedim/}) works as follows:
First start Mathematica. Then read in the input file {\tt liedim.m} taken from \cite{Lo1}.
Then read in an input file which in the case (\ref{eq25}) looks like (from now on we write $b,c,d,e,f,g$ instead
of $B,C,D,E,F,G$).
\medskip

{\tt generators=$\{$b,c,d,e,f,g$\}$}

{\tt gensigns=$\{$1,1,1,1,1,1$\}$}

{\tt relations=$\{$lie[b, c], sq[c]+lie[b,d],sq[d]+lie[b,e],lie[c,e],lie[c,f],sq[e]+lie[d,f],

lie[b,g],lie[e,f]+lie[c,g],lie[d,g],lie[f,g]$\}$}
\medskip

Now a command like {\it e.g.}
\medskip

{\tt maxdegree[7]}
\medskip

\noindent
gives after a few seconds the result 
\medskip

{\tt $\{$ 6, 11, 11, 18, 38, 79, 158$\}$}
\medskip

Here the different numbers are the ranks of the $\eta_i$:s for
$i=1,2,3,\ldots,7$. Thus there are $6$ generators and in degree 2 there
are $11$ generating elements which in the program are denoted by
\medskip

{\tt modbas[2,1],modbas[2,2],...,modbas[2,11]}
\medskip

There is a command {\tt def} with shows how the {\tt modbas}-elements  
look in the Lie algebra $\eta$, so that {\it e.g.} {\tt def[modbas[2,5]]=lie[e,d]} etc.
There is the inverse of that command called {\tt fed}, so that {\tt fed[lie[e,d]] = modbas[2,5]}.
Furthermore there is a command {\tt ideal} which gives the ideal in the big Lie algebra generated in a
certain degree by given {\tt modbas} elements. Thus, for example, {\tt ideal[7,$\{$modbas[3,5]$\}$]} gives the
graded vector space part in $\eta_7$ of the ideal generated by {\tt modbas[3,5]}.
 These commands can be combined with ordinary
Mathematica commands so that the combined command:

\medskip
{\tt For[n=1,n<12,n++,Print[Length[ideal[7,$\{$modbas[3,n]$\}$]][n]]]}   
\medskip

\noindent gives the result:
\medskip

{\tt 1[1]

53[2]

20[3]

15[4]

20[5]

15[6]

1[7]

52[8]

72[9]

52[10]

68[11]}
\medskip
\noindent

 which shows that the element {\tt modbas[3,1]}
generates an ideal which in degree 7 is one-dimensional,
that {\tt modbas[3,2]} generates an ideal which in degree 7 is of dimension $53$ etc.  
In particular  {\tt modbas[3,1]} and {\tt modbas[3,7]} in $\eta_3$ generate
a unique very small ideal in $\eta$ (it can be proved to be of dimension 2 in all degrees $\geq 3$)
 which can be proved to be nilpotent (thus solvable). 
This small ideal will be denoted $rad(\eta)$ since it is a kind of radical of $\eta$
and $\eta/rad(\eta)$ should be ``semisimple'' i.e. a product of ``simple'' Lie algebras.
This is indeed true if we restrict ourselves to ``virtual'' assertions, i.e. for results that are
true in a high degree (in this case in degrees $\geq 3$). 
The word ``virtual'' is inspired by Serre's terminology in \cite[section 1.8]{Ser}.
Let us be more precise:
We have {\tt def[modbas[3,1]]=lie[e, lie[b, b]]}
and {\tt def[modbas[3,7]]=lie[f, lie[f, d]]} so we now study the new Lie algebra ${\overline \eta}=\eta/({\tt lie[e, lie[b, b]],lie[f, lie[f, d]]})$

Now for this Lie algebra we have only $9$ elements in degree 3, denoted by {\tt modbas[3,i]} for $i=1,\ldots 9$.
Note that e.g. {\tt modbas[3,2]} in this new Lie algebra ${\overline \eta}$ is {\tt lie[e, lie[e, b]]}
 whereas in the {\it old} Lie algebra $\eta$ {\tt modbas[3,2]} is {\tt lie[e, lie[d, c]]}.
But a little experimentation shows that 
the $9$ new elements of degree $3$ in ${\bar\eta}$ can now be divided into three parts
 $$
 J11={\tt \{modbas[3,3],modbas[3,5]\}}\quad J12 = {\tt \{modbas[3,2],modbas[3,4]\}}
 $$
and
 $$ 
 J2 = {\tt \{modbas[3,9],modbas[3,8],modbas[3,7],modbas[3,6],modbas[3,1]\}}
 $$ 
These parts are ``orthogonal'' to each other in the sense that the command {\tt ann}
in {\tt liedim} gives that 
\medskip

{\tt ann[J2,3,3] = J11 $\cup$ J12} and {\tt ann[J11 $\cup$ J12,3,3] = J2 }
\medskip

and
\medskip

{\tt ann[J11,3,3] = J12  $\cup$ J2} and {\tt ann[J12,3,3]= J11 $\cup$ J2}
\medskip

Here the command {\tt ann} is the annihilator command, written at the third author's request by the first author:
Type {\tt ?ann} in {\tt liedim}
and the answer is:
\medskip

{\tt ann[a,t,s] gives a basis for the elements in degree s which multiply the

   modbas-elements, which are of degree t, in the list a to zero}
\medskip

Note that multiplication in 
${\overline \eta}$ of e.g. the elements
 {\tt modbas[3,3]} and {\tt modbas[3,5]} in J11 with all of
the elements in  J12 is given as follows
\medskip

{\tt mult[modbas[3,3],J12]} and {\tt mult[modbas[3,5],J12]}
\medskip

\noindent and the result is {\tt $\{$0, 0$\}$} in both cases.
One can continue higher up (cf. section 4)
and define ideals J11, J12 and J2 that are annihilating eachother in all
positive degrees. In section 4, Corollary \ref{cor}, the Hilbert
series of the ideals are calculated.  The result is that the Hilbert
series of the enveloping algebra of   $\bar\eta$ is 
\begin{equation} \label{eq26}
1 \over {(1+t)(1-2t)^2(1-3t+t^2)} 
\end{equation}
It remains to study the extension
\begin{equation} \label{eq27}
0 \longrightarrow rad(\eta)  \longrightarrow \eta \longrightarrow {\overline \eta} \longrightarrow 0 
\end{equation}
Recall that $rad(\eta)$ is generated by {\tt modbas[3,1]=lie[e, lie[b, b]]} and {\tt modbas[3,7]=lie[f, lie[f, d]]}
in $\eta$.
In the following we will abreviate expressions of this form as {\tt ebb} and {\tt ffd}.

 The degree $n$ part of $rad(\eta)$ is given by {\tt ideal[n,$\{$modbas[3,1], modbas[3,7]$\}$]}.
Thus we get

\begin{equation} \label{eq28}
\begin{aligned} 
{\tt ideal[4,\{modbas[3,1], modbas[3,7]\}]}& = {\tt \{modbas[4, 4], modbas[4, 12]\}= \{eebb,ffeb\}}\cr
{\tt ideal[5,\{modbas[3,1], modbas[3,7]\}]}& = {\tt \{modbas[5, 9], modbas[5, 26]\}= \{ebfbb,ffeeb\}}\cr
{\tt ideal[6,\{modbas[3,1], modbas[3,7]\}]}& = {\tt \{modbas[6, 19], modbas[6, 53]\}= \{eebfbb,ffdeeb\}}\cr
{\tt ideal[7,\{modbas[3,1], modbas[3,7]\}]}& = {\tt \{modbas[7, 21], modbas[7, 100]\}= \{ebfbfbb,ffedeeb\}}\cr
{\tt ideal[8,\{modbas[3,1], modbas[3,7]\}]}& = {\tt \{modbas[8, 61], modbas[8, 206]\}= \{eebfbfbb,ffdedeeb\}}\cr
\end{aligned}
\end{equation}
etc. All this indicates that 
$$
rad(\eta)=\bigoplus_{i\geq 3} kc_i \oplus kd_i
$$
where the basis elements $c_i$ and $d_i$ are given in (\ref{eq28}). Furthermore the multiplication in $\eta$ is given
by  ${\tt mult[modbas[s,t],modbas[u,v]]}$. It follows that the ideal $rad(\eta)$ is abelian as far as we calculate
and that the generators $b,c,d,e,f,g$ operate very explitely e.g. as follows:
{\tt op[generators,modbas[7,21]]= $\{$0, 0, 0, modbas[8, 61], 0, 0$\}$} and
{\tt op[generators,modbas[7,100]]= $\{$0, 0, modbas[8, 206], 0, 0, 0$\}$} etc.
and all this seems to subsist in all higher dimensions, and there is a beginning
of a $2$-cocycle $\varphi: {\bar\eta}\times {\bar\eta} \longrightarrow rad(\eta)$
 that describes the extension:
$$
0 \longrightarrow rad(\eta) \longrightarrow \eta \longrightarrow {\bar\eta} \longrightarrow 0
$$
 This is proved in the following way. In section 4 it is proved that the ideal generated by $ebb$
and $ffd$ is of dimension $\leq 2$ in all positive degrees. Next
one puts $L=\bigoplus_{i\geq 3} k R_i \oplus k R_i'$ where the $R_i$ and $R_i'$ are basis elements of degree
$i$. Then one lets the variables $b,c,d,e,f,g$ operate on these elements in a way inspired by the above.
Then one forms the extension ${\bar\eta} \times_{\varphi}L$.
Now there is a natural {\it onto} map from $\eta$ to ${\bar\eta} \times_{\varphi}L$ since the quadratic relations
in $\eta$ are mapped to 0-relations in ${\bar\eta} \times_{\varphi}L$. Thus $rad({\eta})$
which is of dimension $\leq 2$ in all degrees is mapped onto $L$.
Therefore it must be an isomorphism, from which it follows that the Hilbert
series of the
enveloping algebra of $\eta$ is the product of the series of $\bar\eta$ and
the infinite product
given in Theorem \ref{thm1}. The complete details are given in section 4.

\section{Proofs} \label{sec:proofs}

In this section we prove that the Hilbert series of the enveloping algebra of $\eta$ is equal to 
$$
{1\over (1+t)(1-2t)^2(1-3t+t^2)}\prod_{n=2}^{\infty}{(1+t^{2n-1})^2\over(1-t^{2n})^2},
$$ 
from which the last part of Theorem \ref{thm1} follows.

In section \ref{sec:ideal}, it is shown that there are three ideals in $\bar\eta$ which annihilate eachother in all degrees $\geq 3$. It is also shown that they together generate $\bar\eta$ as a vector space in all degrees $\geq 3$. These three subalgebras correspond to the ideals $J11, J12$ and $J2$ in section \ref{sec:irr}.

In section \ref{sec:mod} it is shown that the direct sum of the three subalgebras is isomorphic to $\bar\eta$ as a Lie algebra in degrees $\geq 3$. Explicit presentations of the subalgebras are given, from which the Hilbert series of the enveloping algebra of $\bar\eta$ can be derived. As already mentioned in the beginning of section \ref{sec:irr}, the section also contains three general theorems.

Section \ref{sec:ext} is dedicated to the extension by the radical, which finally gives us the series of the enveloping algebra of $\eta$.

\subsection{Some ideals in $\eta$}\label{sec:ideal}
We start by proving that
the ideal generated by $ebb$ and $ffd$
is of dimension $\leq 2$ in all positive degrees, see Proposition
\ref{rad} below.

We will use the notation $abcd...$ to
denote an iterated Lie expression 
$[a,[b,[c,[d,\ldots]]]]\ldots]$. If $a,b,c,d,\ldots$ are generataors
for a Lie algebra, these iterated Lie expressions generate the Lie
algebra but are not linearly independent. 
Sometimes we also write $a.x$ for $[a,x]$. The
Lie subalgebra (of a given Lie algebra) generated by $a,b,...$ is
denoted $\suba(a,b,\ldots)$. (There is a corresponding function
$\suba[n,\{..\}]$ in {\bf liedim}.) We use $a^2$ to denote the Lie square of
an odd element $a$,
i.e. $a^2 =\frac{1}{2}[a,a]$. Observe that $a^2.x=aax$, but
$x.a^2=\frac{1}{2}xaa$. To avoid misunderstanding, we will not use
expressions like $xaa$ and $xa^2$. Thus we write $aab,bba$ for the two
basis elements in a free Lie algebra on two odd generators $a,b$. We use
$\deg(x)$ to denote the degree of $x$, which is the length of $x$ if $x$
is a string of generators.
\begin{lm}\label{threads} Suppose $x\in\eta$ satisfies
$dx=e^2x=bx=fx=0$ and $y\in\eta$ satisfies
$ey=d^2y=by=fy=c^2y=0$. Define $x_n$ and $y_n$ recursively by
$x_1=x,\ y_1=y$ and 
$$
x_n=\begin{cases}dx_{n-1}&\text{if $n$ is odd}\\
ex_{n-1}&\text{if $n$ is even}\end{cases}\quad\text{and}\quad y_n=\begin{cases}ey_{n-1}&\text{if $n$ is odd}\\
dy_{n-1}&\text{if $n$ is even}\end{cases}
$$
Then

\begin{align*}
&ex_n=dy_n=0\quad\text{if $n$ is even}\\
&dx_n=ey_n=0\quad\text{if $n$ is odd}\\
&b,c^2,bf,f^2 \ \text{ annihilate }\ x_n \ \text{ and }\ y_n \ \text{
for all }\ n\ge1
\end{align*}
\end{lm}
\begin{proof} Using the relations in $\eta$ one may check that the
formulas hold for $n=1,2$. Suppose $n\ge3$ and that
\begin{align*}
&ex_k=dy_k=0\quad\text{for}\quad k<n,\ k \text{ even}\\
&dx_k=ey_k=0\quad\text{for}\quad k<n,\ k \text{ odd}\\
&bx_k=c^2x_{k-1}=(bf)x_{k-1}=f^2x_{k-1}=0\quad\text{for}\quad k<n\\
&by_k=c^2y_{k-1}=(bf)y_{k-1}=f^2y_{k-1}=0\quad\text{for}\quad k<n\\
\end{align*}
We claim that the same formulas also hold for $k=n$. We do the proof
for $x_n$, the proof for $y_n$ is similar. Observe that $e^2d=(bf)e$
which is used twice below, for $ex_n$ and $(bf)x_{n-1}$. It is also
used that
$(bf)d=e^2b$ and $(f^2d)e=f^2e=0$. 
\begin{align*}
n \text{ odd: }&dx_n&=&d^2x_{n-1}=-(be)x_{n-1}=-bex_{n-1}-ebx_{n-1}=0\\
n \text{ even, $n\ge4$: }&ex_n&=&e^2dx_{n-2}=(e^2d)x_{n-2}+deex_{n-2}=((bf)e)x_{n-2}=\\
&&=&(bf)ex_{n-2}-e(bf)x_{n-2}=0\\
n \text{ odd, $n\ge$3: }&bx_n&=&(bd)x_{n-1}-0=-c^2ex_{n-2}=-ec^2x_{n-2}=0\\
n \text{ even: }&bx_n&=&(be)x_{n-1}-0=-d^2x_{n-1}=-ddx_{n-1}=0\\
n \text{ odd, $n\ge$3: }&c^2x_{n-1}&=&(c^2e)x_{n-2}+0=0\\
n \text{ even, $n\ge$4:
}&c^2x_{n-1}&=&(c^2d)x_{n-2}+0=(d^2b)x_{n-2}=\\
&&=&d^2bx_{n-2}-bd^2x_{n-2}=-bdx_{n-1}=0\\
n \text{ odd, $n\ge$3:
}&(bf)x_{n-1}&=&((bf)e)x_{n-2}+0=e^2dx_{n-2}-de^2x_{n-2}=-dex_{n-1}=0\\
n \text{ even, $n\ge$4:
}&(bf)x_{n-1}&=&((bf)d)x_{n-2}+0=e^2bx_{n-2}-be^2x_{n-2}=0\\
n \text{ odd, $n\ge$3:
}&f^2x_{n-1}&=&(f^2e)x_{n-2}+0=0\\
n \text{ even, $n\ge$4:
}&f^2x_{n-1}&=&((f^2)d)x_{n-2}+0=((f^2d)e)x_{n-3}-e(f^2d)x_{n-3}=\\
&&=&0-ef^2dx_{n-3}+edf^2x_{n-3}=0
\end{align*}
\end{proof}
The lemma may be applied to $x=g, x=bbe$ or $y=c, y=c^2$ or $y=ffd$. It
follows from the lemma that also $d^2x_n=e^2x_n=d^2y_n=e^2y_n$ for all
$n\ge1$. If we define $d'=d+e$, it follows that $x_n,y_n$ may be
obtained from $x,y$ by succesive multiplication by $d'$. 

We will now
prove the claim from the previous section that the ideal $\rad (\eta)$
is abelian and at most two-dimensional in each degree $\ge3$. We will also
include $c^2$ in the ideal, which makes no difference in degrees
$\ge3$, since $b,c,e,f,g$ annihilate $c^2$ and $dc^2=bbe$. 
\begin{prp}\label{rad}
The ideal $\rad (\eta)$ in $\eta$ generated by $c^2$ and $ffd=cd'g$ is
abelian and linearly spanned by the
elements $r_{n}=(d')^{n-2}c^2$, $n\ge2$ and $r_{n}'=(d')^{n-3}cd'g$, $n\ge3$.
\end{prp}
\begin{proof}
Lemma
\ref{threads} may be applied and we get
$br_n=br_n'=e^2r_n=e^2r_n'=d^2r_n=d^2r_n'=0$ for all $n\ge3$. 

We
now claim that $cr_n=0$ for $n\ge2$. 

This is proved by induction. For
$n=2$ we get $c.c^2=0$. Suppose the claim holds for all $k<n$.
\begin{align*}
cr_n=cd'r_{n-1}=(cd')r_{n-1}-0=((cd')d')r_{n-2}+d'(cd')r_{n-2}
\end{align*}
The last term above is zero by induction, since it is equal to
$d'cr_{n-1}+d'd'cr_{n-2}$. We may continue in the same way. In
the next step we use that $((cd')d').r_{n-3}=0$, which is true by
induction developing the expression as an iterated action on $r_{n-3}$
and using
$c.r_i=0$ for $i<n$ and $d'.r_i=r_{i+1}$. In the end we get
$$
\ldots(((((cd')d')d')d'\ldots)d').c^2
$$
With notation from Lemma \ref{threads}, using the commutative law,
this may be seen to be $\pm y_{n-1}.c^2$, where $y_1=c$. Hence the expression
is zero by the lemma, which proves the claim. 

In the same way it is proved that $gr_n=0$ for $n\ge2$. Here we apply
Lemma \ref{threads} with $x_n$ defined from $x_1=g$. 

Also in the same way it is proved that $cr_n'=gr_n'=0$. Here we use
Lemma \ref{threads} to get $(ffd)y_n=(ffd)x_n=0$. 

We have, by Lemma \ref{threads}, that $dr_n$ is either zero or $r_{n+1}$, the same is true for $r_n'$
and for the action by $e$. Hence it only remains to prove that the
action by $f$ is zero, which is now easy to prove by induction, since
$fd'=e^2+cg$.

To prove that the ideal is abelian, we prove by induction over $n$
that $[r_n,r_m]=[r_n,r_m']=0$ for all $n\ge2$ and $m\ge3$ and that
$[r_n',r_m']=0$ for all $n,m\ge3$. The induction start follows since,
$c^2r_m=c^2r_m'=0$ and $(ffd)r_m'=0$ 
by Lemma \ref{threads}. The induction step is easy:
$$
[r_n,r_m]=[d'r_{n-1},r_m]=d'.[r_{n-1},r_m]\pm [r_{n-1},r_{m+1}]
$$
and similarly for $[r_n,r_m']$ and $[r_n',r_m']$.

\end{proof}
We will now prove that the elements of degree $\ge3$ in the subalgebra
of $\eta$
generated by $d,e$, denoted $\suba(d,e)_{\ge3}$, is an ideal
modulo $\rad (\eta)$. To do this we first prove a lemma. 
\begin{lm}\label{cg} We have $cx=gx=0$ for all $x\in
\suba(d,e)_{\ge3}$.
\end{lm} 
\begin{proof}
Proof by induction. We have $cdde=geed=0$ since
$c.d^2=g.e^2=ce=gd=0$. Also $ceed=e^2dc$, $gdde=d^2eg$ which are
zero by Lemma \ref{threads}. It is enough to consider $x$ as an
iterated Lie product of $d$:s and $e$:s. The induction step for $c$ is
easy if $x=ex'$. Otherwise, by induction, $cx=(cd)x'$. Since
$cd^2=0$, we may suppose $x=dex''$. Using Lemma \ref{threads} we see
that the only case we need to consider is when $x$ never has a
repetition of two d:s or e:s. But $x$ starts to the right with $dde$ or
$eed$. Hence the result follows for $c$, and $g$ is handled in the same
way. 
\end{proof}
\begin{prp}\label{de}
We have that $\suba(d,e)_{\ge3}+\rad (\eta)$ is an ideal in $\eta$.
\end{prp}
\begin{proof}
By Lemma \ref{cg} it is enough to consider the action by $b$ and
$f$. The proof is by induction. We have $bdde=edc^2\in \rad (\eta)$ and $beed=-ddde$. Also
$feed=-dffd\in \rad (\eta)$ and $fdde=-eeed$. We have
$bex=-d^2x-ebx$ and $fdx=-e^2x-dfx$ which stay in
$\suba(d,e)_{\ge3}+\rad (\eta)$ by induction. Also $bdx=-c^2x-dbx=-dbx$
and $fex=-(cg)x-efx=-efx$ by Lemma \ref{cg} and we are done.  

\end{proof}
\begin{prp}\label{bf}
We have that $\suba(b,f)_{\ge3}+\rad (\eta)$ is an ideal in $\eta$ and
$\suba(d,e).\suba(b,f)_{\ge3}\subset \rad (\eta)$, more precisely, let
$r_n,r_n'\in\rad(\eta)$ be defined as in Proposition \ref{rad} and
define $b_n,f_n\in\suba(b,f)$ recursively by 
 $b_2=b^2$, $f_2=f^2$
$$
b_n=\begin{cases}fb_{n-1}&\text{if $n$ is odd}\\
bb_{n-1}&\text{if $n$ is even}\end{cases}\quad\text{and}\quad
f_n=\begin{cases}bf_{n-1}&\text{if $n$ is odd}\\ 
ff_{n-1}&\text{if $n$ is even}\end{cases}
$$
 Then $dx=ex=0$ for any iterated Lie product of length at least $4$
in $\suba(b,f)$ which is different from $b_n$ and $f_n$
for all $n$. Also for $n$ even, $n\ge4$, $ef_n=db_n=0$ and $df_n=-r_{n+1}'$,
$eb_n=r_{n+1}$  and for $n$ odd, $df_n=eb_n=0$ and $ef_n=r_{n+1}'$,
$db_n=-r_{n+1}$.

\end{prp}
\begin{proof}
Put $J=\suba(b,f)_{\ge3}$. We have $c.J=g.J=0$. Thus it is enough to
prove the second claim. We have $ebbf=dffb=0$ and $ef_3=-effb=r_4'$, $db_3=-dbbf=
-r_4$. Hence the claim follows for $n=3$. Suppose $n\ge4$ and the claim
is true for all $k<n$. We have by induction and Lemma \ref{threads} 
\begin{align*}
n \text{ even: }&db_n=dbb_{n-1}=-c^2b_{n-1}-bdb_{n-1}=br_{n}=0\\
&eb_n=ebb_{n-1}=-d^2b_{n-1}-beb_{n-1}=dr_n=r_{n+1}\\
&df_n=dff_{n-1}=-e^2f_{n-1}-fdf_{n-1}=-er_{n}'=-r_{n+1}'\\
&ef_n=eff_{n-1}=-(cg)f_{n-1}-fef_{n-1}=-fr_n'=0\\
n \text{ odd, $\ge5$: }&db_n=dfb_{n-1}=-e^2b_{n-1}-fdb_{n-1}=-er_{n}=-r_{n+1}\\
&eb_n=efb_{n-1}=-(cg)b_{n-1}-feb_{n-1}=-fr_n=0\\
&df_n=dbf_{n-1}=-c^2f_{n-1}-bdf_{n-1}=br_{n}'=0\\
&ef_n=ebf_{n-1}=-d^2f_{n-1}-bef_{n-1}=dr_n'=r_{n+1}'
\end{align*}

\end{proof}
It follows from Proposition \ref{bf} that $b_n,f_n$ is not in the
linear span of all iterated Lie products of length $n$ different from $b_n,f_n$ in
$\suba(b,f)$, if we knew that $r_n\ne0$, $r_n'\ne0$ for all $n$. We
can prove this without the assumption of the non-vanishing of
$r_n,r_n'$ and this result will be needed in section \ref{sec:ext}.

\begin{prp}\label{zero} Let $B_n,F_n$ be elements in the free Lie
algebra $\F(B,F)$ defined as in Proposition \ref{bf}. Then the linear
span of all iterated Lie products of length $\ge3$ which are not of
the form $B_n,F_n$ for any $n$ is an ideal in $\F(B,F)$ and also $B_n$
and $F_n$ are linearly independent modulo this ideal.
\end{prp}
\begin{proof} Define a graded module $R$ over $\F(B,F)$, where $B,F$
has degree 1, by letting $R_n$ be a
basis of $R$ in degree $n$ for all $n\ge1$ and $B.R_n=F.R_{n+1}=0$ if
$n$ is even and $B.R_n=R_{n+1}$ if $n$ is odd and $F.R_n=R_{n+1}$ if
$n$ is even, i.e., we have a ''thread"
$$
R_1\xrightarrow{B}R_2\xrightarrow{F}R_3\xrightarrow{B}R_4\xrightarrow{F}\ldots
$$
Define two derivations (of degree 0 and 1 respectively) $\phi_B,\phi_F$ from $\F(B,F)$ to R by,
$\phi_B(B)=R_1, \phi_B(F)=0$ and $\phi_F(B)=0, \phi_F(F)=R_2$ (it is a
wellknown fact
that a derivation on a free Lie algebra is defined uniquely by the 
definition on the generators). Then $\phi_B(BF)=\phi_F(BF)=0$ and
$\phi_B(BBF)=-R_3,\ \phi_F(BBF)=0,\ \phi_B(FFB)=0,\
\phi_F(FFB)=-R_4$. If $\deg(X)\ge3$ then 
\begin{align*}
\phi_B(FX)&=F\phi_B(X)\\
\phi_B(BX)&=\pm XR_1 + B\phi(X)=B\phi_B(X)\\
\phi_F(FX)&=\pm XR_2 -F\phi_F(X)=-F\phi(X)\\
\phi_F(BX)&= -B\phi_F(X)\\
\end{align*}
Hence $\phi_B,\phi_F$ are linear in degree $\ge3$. Let
$I=\ker(\phi_B)\cap\ker(\phi_F)$. Then $I$ is an ideal and, since $\phi_B(B_n)=R_n$ and
$\phi_F(F_n)=\pm R_{n+1}$, it is easy to see that $I$ in degree $\ge3$
is the linear span of all iterated Lie products which are not of the
form $B_n,F_n$. That $B_n$ and $F_n$ are linearly independent modulo
$I$ folows easily from the fact that they have different bidegree in
$\F(B,F)$ where $\bideg(B)=(1,0)$ and $\bideg(F)=(0,1)$.
\end{proof} 

\vspace{10pt}
Next, we will prove that $\suba(c,d',g)_{\ge3}$ is an ideal. To do this, we start with a lemma.
\begin{lm}\label{dd}
We have $(\{d^2,e^2\}\cup\suba(d,e)_{\ge3}).\suba(c,d',g)_{\ge3}=0$.
\end{lm}
\begin{proof}
According to Lemma \ref{cg} we have
$c.\suba(d,e)_{\ge3}=g.\suba(d,e)_{\ge3}=0$. Also
$d'.\suba(d,e)_{\ge3}\subset\suba(d,e)_{\ge3}$. Since a non-zero Lie
expression in $\suba(c,d',g)_{\ge3}$ must contain at least one $c$ or
$g$, it follows that $\suba(c,d',g)_{\ge3}.\suba(d,e)_{\ge3}=0$. We
prove by induction that 
$d^2.\suba(c,d',g)_{\ge3}=e^2.\suba(c,d',g)_{\ge3}=0$. The induction
start is somewhat tedious but manageable. We have
$d^2c=d^2g=e^2c=e^2g=0$. Hence for the induction step we only have to
consider the cases $d^2.d'x$ and $e^2.d'x$. We have
\begin{align*}
d^2.d'x=(d^2d').x+d'.d^2.x=(d^2e).x=0\\
e^2.d'x=(e^2d').x+d'.e^2.x=(e^2d).x=0
\end{align*}
where we have used that $\suba(d,e)_{\ge3}.\suba(c,d',g)_{\ge3}=0$.
\end{proof}

\vspace{10pt}
The following lemma will be needed in the next section. In fact, a
stronger result is true, namely $b.\suba(c,d',g)_{\ge3}=0$, but we do not need this.

\begin{lm}\label{b}
We have that $b.\suba(c,d',g)_{\ge3}\subset\rad(\eta)$.
\end{lm}
\begin{proof} Proof by induction. The induction start consists of
checking 7 equalities. We have $b.c=b.g=0$ and $b.d'=-c^2-d^2$. The
result follows from Lemma \ref{dd} and since $c^2\in\rad(\eta)$. 
\end{proof}
\begin{prp}\label{cd'g}
We have that $\suba(c,d',g)_{\ge3}$ is an ideal in $\eta$ and
$\suba(b,f)_{\ge3}.\suba(c,d',g)_{\ge3}=0$.
\end{prp}
\begin{proof}
By Proposition \ref{bf}, $\suba(b,f)_{\ge3}+\rad(\eta)$ is an
ideal. Since it is annihilated by $c$ and $g$, it follows in the same
way as in the proof of the previous lemma that
$\suba(c,d',g)_{\ge3}.\suba(b,f)_{\ge3}=0$. By Lemma \ref{b}, $b.x\in\rad(\eta)_{\ge3}\subset
\suba(c,d',g)_{\ge3}$ for $x\in 
\suba(c,d',g)_{\ge3}$. We now prove by induction
that $f.x\in \suba(c,d',g)_{\ge3}$ for $x\in
\suba(c,d',g)_{\ge3}$. Again the induction start consists of checking
a number of equalities. We have $fc=fg=0$ so the
only case to consider is $f.d'x$. But 
$fd'=-e^2-cg$. The claim now follows by Lemma \ref{dd}. Finally we
prove by induction that $d.x, e.x\in \suba(c,d',g)_{\ge3}$ for $x\in
\suba(c,d',g)_{\ge3}$. But $dc=d'c$, $dg=0$, $ec=0$ and
$eg=d'g$. Thus we only have to consider $d.d'x$ and $e.d'x$. But
$dd'=2d^2+de=(d')^2+d^2-e^2$ and $ed'=2e^2+de=(d')^2-d^2+e^2$. The
result now follows by Lemma \ref{dd}. 
 
\end{proof}

\begin{prp}\label{sum}
For each $n\ge3$ we have $\eta_n=\suba(c,d',g)_n+\suba(d,e)_n+\suba(b,f)_n$.
\end{prp}
\begin{proof}
The claim is true for $n=3$ by direct calculation. Suppose it is true
for all $i<n$. Since $\rad(\eta)$ is generated as an ideal by $c^2$ and $cd'g$ it
follows from Proposition \ref{cd'g} that
$\rad(\eta)\subset\suba(c,d',g)$. We have
$\eta_1.\suba(c,d',g)_{n-1}\subset\suba(c,d',g)_n$. By Proposition
 \ref{de}, \ref{bf} we also have
$\eta_1.\suba(d,e)_{n-1}\subset\suba(d,e)_n+\rad(\eta)$ and 
$\eta_1.\suba(b,f)_{n-1}\subset\suba(b,f)_n+\rad(\eta)$. 
Since $\eta_n=\eta_1.\eta_{n-1}$, the claim now follows.
\end{proof}

\subsection{A module over $\bar\eta$} \label{sec:mod}
Our goal in this section is to prove that $\bar\eta=\eta/\rad(\eta)$
is isomorphic in degree $\ge3$ to a direct sum of three Lie algebras,
which are $\F(C,D',G)/<C^2,[C,[D',G]]>$, $\F(D,E)$ and $\F(B,F)$ all taken
in degree $\ge3$, where e.g., $\F(D,E)$ means the free Lie algebra on
$D,E$.  To do this, we will form the direct sum and make it
to an $\bar\eta$-module and map it surjectively to $\bar\eta$. Then we
will prove three theorems which is an application of Theorem 5.3 in \cite{Lo-Ro} for
our situation.

The definition of the module structure will be done by defining the
action of the generators as derivations and then checking that the
relations in $\bar\eta$ act as zero. We will use the wellknown fact
that a derivation on a free Lie algebra is defined uniquely by 
defining the derivation on the generators and these definitions are
arbitrary. 

Here follows the definition of the module $M$ as a vector
space. Observe that the elements $B_2, D_c,\ldots$ are just symbols and
has nothing to do with the elements in $\bar\eta$, but we will soon
define a map to $\bar\eta$ such that e.g., $B_2$ is mapped to $b^2$
and $D_c$ is mapped to $[d,c]$ modulo $\rad(\eta)$.

\begin{dfn}\label{dfn}
\begin{align*}
M_1&=\spa(B_1,C_1,D_1,E_1,F_1,G_1)\\ 
M_2&=\spa(B_2,D_c,D_2,E_d,B_F,E_2,F_2,G_c,G_e,G_2)\\
M_{\ge3}&= \bigl(\F(C,D',G)/<C^2,[C,[D',G]]>\oplus\F(D,E)\oplus\F(B,F)\bigr)_{\ge3}
\end{align*}
\end{dfn}
The next step is to define an action of $\F(b,c,d,e,f,g)$ on $M$. This
is done by defining the action of each generator, explicit on
$M_1,M_2$ and also explicit on the elements $\{C,D',G,D,E,B,F\}$ and
then checking that the extended derivation maps the ideal
$<C^2,[C,[D',G]]>$ to itself.

\vspace{10pt} 
$$
\begin{array}{c|c|c|c|c|c|c|}
&B_1&C_1&D_1&E_1&F_1&G_1\\
\hline
b&2B_2&0&0&-D_2&B_F&0\\
\hline
c&0&0&D_c&0&0&G_c\\
\hline
d&0&D_c&2D_2&E_d&-E_2&0\\
\hline
e&-D_2&0&E_d&2E_2&-G_c&G_e\\
\hline
f&B_f&0&-E_2&-G_c&2F_2&0\\
\hline
g&0&G_c&0&G_e&0&2G_2\\
\hline
\end{array}
$$

\vspace{10pt}
Im the following table we use the notation (for use in the next section) $D_3=[E,D^2]$,
x$E_3=[D,E^2]$, $B_3=[F,B^2]$, $F_3=[B,F^2]$, $C_3=D'D'C$ and
$G_3=D'D'G$.

\vspace{10pt}
$$
\begin{array}{c|c|c|c|c|c|c|c|c|c|c|}
&B_2&D_c&D_2&E_d&B_f&E_2&F_2&G_c&G_e&G_2\\
\hline
b&0&0&0&0&-B_3&D_3&F_3&0&0&0\\
\hline
c&0&0&0&-C_3&0&0&0&0&0&-GGC\\
\hline
d&0&0&0&-D_3&D_3&E_3&0&D'CG&G_3&0\\
\hline
e&0&C_3&D_3&-E_3&E_3&0&0&0&0&-GGD'\\
\hline
f&F_3&0&E_3&D'CG&-F_3&0&0&0&GGC&0\\
\hline
g&0&GD'C&0&-G_3&0&0&0&GGC&GGD'&0\\
\hline
\end{array}
$$

\vspace{10pt} The next table shows how the generators $b,c,d,e,f,g$
operate on the generators $C,D',G,D,E,B,F$ defining $M_{\ge3}$.
$$
\begin{array}{c|c|c|c|c|c|c|c|}
&C&D'&G&D&E&B&F\\
\hline
b&0&0&0&0&-D^2&2B^2&BF\\
\hline
c&0&CD'&CG&0&0&0&0\\
\hline
d&CD'&(D')^2&0&2D^2&DE&0&0\\
\hline
e&0&(D')^2&D'G&DE&2E^2&0&0\\
\hline
f&0&-CG&0&-E^2&0&BF&2F^2\\
\hline
g&CG&D'G&2G^2&0&0&0&0\\
\hline
\end{array}
$$

\vspace{10pt}
We now check that the generators map the ideal $I=<C^2, [C,[D',G]]>$ to
itself. We have
\begin{align*}c.C^2&=0,\ 
c.[C,[D',G]]=-[C,[[C,D'],G]]+[C,[D',[C,G]]]=-[C,[C,[D',G]]]\in I\\
d.C^2&=[[D',C],C]\in I\\
d.[C,[D',G]]&=[[D',C],[D',G]]-[C,[(D')^2,G]]=[[C,[D',G]],D']\in I\\
e.C^2&=0,\ e.[C,[D',G]]=-[C,[(D')^2,G]]+[C,[D',[D',G]]]=0\\
f.C^2&=0,\ f.[C,[D',G]]=-[C,[(D')^2,G]]+[C,[D',[D',G]]]=0\\
g.C^2&=[[C,G],C]\in I\\
g.[C,[D',G]]&=[[C,G],[D',G]]-[C,[[D',G],G]]+[C,[D',[G,G]]]=[[C,[D',G]],G]\in
I
\end{align*}
Now $M$ is a module over the free Lie algebra on $b,c,d,e,f,g$. The
next step is to show that all the relations in $\bar\eta$ are
operating as zero in all degrees. In degree $\le2$ this consists of checking a finite
number of equalities. We will do this just in one case, $(e^2+df).E_d$,
\begin{align*}(e^2+df).E_d&=e.EED+d.D'CG+f.DDE\\
&=(e.E^2)D+EEED+D'D'CG-D'(CD')G-(E^2D)E=D'CD'G=0
\end{align*}
We now prove that the relations operate as zero in
degrees $\ge3$. Since the generators act as derivations in degree
$\ge3$ and since a Lie
product of a derivation is again a derivation, it is enough to check
that the relations act as zero on the generators $C,D',G,D,E,B,F$ and
this consists also of checking a finite number of equalities. We will anyway present the
computations.  

We start with the operation on the first Lie algebra
generated by $C,D',G$. Any
monomial containing a $b$ will operate as zero. Also $c,g$ operate
naturally as $\ad_C,\ad_G$ and $d+e$ operates as $\ad_{D'}$. Moreover
it is easy to see that $c^2,d^2,e^2,f^2$ operate as zero,
e.g. $e^2.G=e.D'G=D'D'G-D'D'G=0$. 
\begin{align*}(ce).C&=0,\ (ce).D'=C(D')^2+e.CD'=0,\\
(ce).G&=c.D'G+e.CG=CD'G-CD'G=0\\
(cf).C&=0,\ (cf).D'=-CCG+f.CD'=-CCG-CCG=0,\ (cf).G=f.CG=0\\
(df).C&=f.D'C=-(CG)C=0,\\
(df).D'&=-d.CG+f.(D')^2=-(CD')G-(CG)D'=\\
&=-CD'G+(CG)D'-(CG)D'=0,\ (df).G=0\\
(ef+cg).C&=c.GC=CCG=0,\ (ef+cg).D'=-e.CG+f.(D')^2+(CG).D'=\\
&=CD'G-(CG)D'+(CG)D'=0,\\ 
(ef+cg).G&=f.D'G+(CG)G=-(CG)G+(CG)G=0\\
(dg).C&=d.GC+g.CD'=-GCD'+GCD'=0,\\
(dg).D'&=d.GD'+g.(D')^2=-G(D')^2+G(D')^2=0,\
(dg).G=d.GG=0\\
(fg).C&=f.GC=0,\ (fg).D'=f.GD'+g.(D')^2=-G(D')^2+G(D')^2=0,\\
(fg).G&=f.GG=0\\
(ffd).C&=(ffd).D'=(ffd).G=0\quad\text{since $f^2$ operates as zero}
\end{align*}
We continue with the operation of the relations on $\F(D,E)$. We will use
the fact that $c,g$ operate as zero and $d,e$ operate naturally as $\ad_D,\ \ad_E$. 
\begin{align*}
(bd).D&=b.DD=0,\ (bd).E=b.DE-d.D^2=-DD^2-DD^2=0\\
(d^2+be).D&=d.DD+b.ED=DDD+DDD=0\\
(d^2+be).E&=DDE+b.EE-e.D^2=DDE-2DDE+DDE=0\\
(e^2+df).D&=EED-d.E^2+f.DD=EED+EED-2EED=0,\\
(e^2+df).E&=f.DE=-EEE=0\\
(ef).D&=-e.E^2+f.ED=-EE^2=0,\ (ef).E=f.EE=0\\
(ffd).D&=f.f.DD=-2f.EED=-2EEE^2=0,\ (ffd).E=f.f.DE=-f.EEE=0
\end{align*}
Finally we study the operations on $\F(B,F)$. Here $c,d,e,g$ operate as
zero so there is nothing to prove.

\vspace{10pt} Define a map $\phi_1$, by $\phi_1(C)=c,\
\phi_1(D')=d+e,\ \phi_1(G)=g$. Since $c^2=c(d+e)g=0$ in $\bar\eta$,
this induces a Lie algebra map 
$L_1=\F(C,D',G)/<C^2,[C,[D',G]]>\to \bar\eta$. Similarly Lie algebra
maps $\phi_2:\ L_2=\F(D,E)\to \bar\eta$ and $\phi_3:\ L_3=\F(B,F)\to\bar\eta$
are defined by $\phi_2(D)=d,\
\phi_2(E)=e$ and $\phi_3(B)=b,\ \phi(F)=f$. By the construction above
we have that $L_1,L_2,L_3$ are $\bar\eta$-module Lie algebras, but
$\phi_i$ is \emph{not} an $\bar\eta$-module homomorphism,
$i=1,2,3$, (where we consider $\bar\eta$ as a module over itself via
the $\ad$-representation), e.g., $0=\phi_2(c.D)\ne c\phi(D)=cd$.  

Define a map $\phi: \ M\to\bar\eta$ by
$\phi(B_1)=b,\ \phi(C_1)=c,\ \phi(D_1)=d,\ \phi(E_1)=e,\ \phi(F_1)=f,\
\phi(G_1)=g$
and
\begin{align*}
\phi(B_2)&=b^2,\ \phi(D_c)=[c,d],\ \phi(D_2)=d^2,\ \phi(E_d)=[d,e],\
\phi(B_f)=[b,f]\\
\phi(E_2)&=e^2,\ \phi(F_2)=f^2,\ \phi(G_c)=[c,g],\ \phi(G_e)=[e,g],\
\phi(G_2)=g^2
\end{align*}
and in degree $\ge3$, $\phi$ is defined as
$\phi_1+\phi_2+\phi_3$. Then $\phi$ is surjective in degree $\le2$ by
inspection and also in degree $\ge3$ by Proposition \ref{sum}. To
prove that $\phi$ is an $\bar\eta$-module homomorphism we need the
following lemma.
\begin{lm}\label{ho}
Suppose $X\in (L_i)_{\ge3}$ and suppose $H=C$ or $H=D'$ or $H=G$ if
$i=1$ and $H=D$ or $H=E$ if $i=2$ and $H=B$ or $H=F$ if $i=3$. Suppose
$a\in\{b,c,d,e,f,g\}$. Then
$$
(\phi(a.H)-a\phi(H))\phi(X)=0
$$
\end{lm}
\begin{proof}
If $\phi(a.H)=a\phi(H)$ there is nothing to prove. By looking in the table
above, we are hence left with the following combinations 
$b.D',d.D',e.D',f.D',c.D,f.E,g.E,d.F,e.B,e.F$.
\begin{align*}
&X\in (L_1)_{\ge3}:\\
&(\phi(b.D')-b(d+e))\phi(X)=-be\phi(X)=d^2\phi(X)=0,\text{ by Lemma \ref{dd}}\\
&(\phi(d.D')-d(d+e))\phi(X)=((d+e)^2-2d^2-de)\phi(X)=(e^2-d^2)\phi(X)=0,\text{ by Lemma \ref{dd}}\\
&(\phi(e.D')-e(d+e))\phi(X)=((d+e)^2-2e^2-de)\phi(X)=(d^2-e^2)\phi(X),\text{ by Lemma \ref{dd}}\\
&(\phi(f.D')-f(d+e))\phi(X)=(-cg+e^2-ef)\phi(X)=e^2\phi(X)=0,\text{ by
Lemma \ref{dd}}\\
&X\in (L_2)_{\ge3}:\\
&(\phi(c.D)-cd)\phi(X)=-cd\phi(X)=0,\text{ by Lemma \ref{cg}}\\
&(\phi(f.E)-fe)\phi(X)=-fe\phi(X)=cg\phi(X)=0,\text{ by Lemma
\ref{cg}}\\
&(\phi(g.E)-ge)\phi(X)=-eg\phi(X)=0,\text{ by Lemma \ref{cg}}\\
&X\in (L_3)_{\ge3}:\\
&(\phi(d.F)-df)\phi(X)=-df\phi(X)=0,\text{ by Proposition \ref{bf}}\\
&(\phi(e.B)-eb)\phi(X)=-eb\phi(X)=0,\text{ by Proposition \ref{bf}}\\
&(\phi(e.F)-ef)\phi(X)=-ef\phi(X)=0,\text{ by Proposition \ref{bf}}\\
\end{align*}

\end{proof}

\begin{prp}\label{homo}
The map $\phi: M\to \bar\eta$ is an $\bar\eta$-module homomorphism
\end{prp}
\begin{proof}
As usual there is a number of equalities to check in degrees
$\le3$. For higher degrees we use induction. Suppose $a$, $H$ and $X$ are
as in Lemma \ref{ho} and suppose we know inductively
$\phi(a.X)=a\phi(X)$ and we want to prove that
$\phi(a.HX)=a\phi(HX)$. We have
\begin{align*}
&\phi(a.HX)=\phi((aH)X)-\phi(H(a.X))=\phi(a.H)\phi(X)-\phi(H)\phi(a.X)=\\
&=\phi(a.H)\phi(X)-\phi(H)a\phi(X)=(a\phi(H))\phi(X)-\phi(H)a\phi(X)=a\phi(H)\phi(X)=a\phi(HX)
\end{align*}
\end{proof}
We will now prove two properties about the module $M$, which are
important for us to be able to deduce that $\phi$ is an isomorphism.
\begin{lm}\label{orto}
\begin{align*}
\phi_i((L_i)_{\ge3}).(L_j)_{\ge3}=0\quad\text{for}\ \ i\ne j\\
\\
\phi_i(x).y=[x,y]\quad\text{for}\ \ x,y\in L_i
\end{align*}
\end{lm}
\begin{proof}
Since $c,d,e,g$ act trivially on $L_3$ we have that
$\phi_i((L_i))(L_3)=0$ for $i=1,2$. Since $b^2,f^2,c,g$ act trivially on
$L_2$, we have $\phi_i((L_i)_{\ge3})(L_2)=0$ for $i=1,3$. Since
$b,d^2,e^2,f^2$ act trivially on $L_1$, we have
$\phi_i((L_i)_{\ge3})(L_1)=0$ for $i=2,3$.

To prove the second claim, we have already noticed that $c,d+e,g$ act
on $L_1$ as $\ad_C,\ad_{D'},\ad_G$ and $d,e$ act on $L_2$ as
$\ad_D,\ad_E$ and $b,f$ act on $L_3$ as $\ad_B,\ad_F$. Hence the claim
is true when $x$ is of degree 1. The induction step is as follows
($\phi(H)=h$, we use $x$ in the exponent to denote the degree of $x$)
\begin{align*}
\phi(Hx).y=h.\phi(x).y-(-1)^x\phi(x).h.y=h.[x,y]-(-1)^x[x,h.y]=[h.x,y]=[Hx,y]
\end{align*}
\end{proof}
We now give a general theorem which we will use to prove that the map
$\phi:\ M\to\bar\eta$ is an isomorphism.

\begin{thm}\label{modul}
Let $\g$ be a graded (in degrees $\ge1$) Lie superalgebra generated by a set $X$ modulo a
set of relations $Y$ in the free Lie algebra $\F(X)$. Let
$M=\oplus_{i\ge n}M_i$ be a graded $\g$-module, generated by
$M_n$. Let $\phi:M\to\g_{\ge n}$ be a graded surjective $\g$-module
homomorphism, such that 
\begin{equation}\label{modmod}
\phi(x).y=-(-1)^{xy}\phi(y).x\quad \text{for all}\quad x\in M_n\text{
and all }y\in M
\end{equation}
Suppose $\deg(x)\le n$ for all $x\in X$ and all $x\in Y$. Suppose
$\phi$ is an isomorphism in all degrees $n\le i\le 2n-1$. Then $\phi$
is an isomorphism in all degrees. 
\end{thm}
\begin{proof}
Put $\hat M=G_1\oplus\ldots\oplus G_{n-1}\oplus M$ where $G_i$ is an isomorphic copy
of $\g_i$ for $i=1,\ldots n-1$. Extend the map $\phi$ to a map of vector
spaces $\phi:\hat M\to \g$. We define an action of $X$ on $\hat M$ by mirroring the action
on $\g$ up to degree $n-1$ and then in degree $\ge n$ using the $\g$-module structure on
$M$. Then the relations in $\g$ will operate as zero
in degree $\ge n$ and also in degree $\le n-1$, since $\deg(y)\le n$
for all relations $y\in Y$ and by assumption
$\phi$ is an isomorphism in degree $\le 2n-1$. Hence $\hat M$ is a
$\g$-module and $\phi$ is a surjective $\g$-module map. Let
$m_x\in\hat M$ be elements such that $\phi(m_x)=\bar x$ for all $x\in X$.
We have that $\hat M$ 
is generated by $\{m_x;x\in X\}$ as a $\g$-module, since this is true in degree
$\le n$ and in higher degrees it follows since $M$ is generated by
$M_n$ as a $\g$-module. Now Theorem 5.3 in \cite{Lo-Ro} may be 
applied. Since $\phi$ is an isomorphism in degree $\le n$, the relations
do not give any elements in the kernel. Hence, we are left with the commutative law,
$\phi(a).m_x=-(-1)^ax.a$, where $a\in M$ and $x\in X$.  Since $\phi$
is an isomorphism in degree $\le 2n-1$, it is enough to assume that
$\deg(a)\ge n$. 

By \eqref{modmod} and Lemma 5.2 in \cite{Lo-Ro}, the
commutative law  
\begin{equation}\label{modcomm}
\phi(a).y=-(-1)^{ay}\phi(y).a
\end{equation}
 holds when $y$ also is of
degree $\ge n$. 

We want to
reduce the degree of $y$ down to 1, and we do this by induction
downwards. Suppose \eqref{modcomm} is true when
the degree of $y$ is $>k$ and assume $y$ is of degree $k<n$. Since
$\phi$ is an isomorphism in degree $\le 2n-1$, the law is true when the
degree of $a$ is $n$. We prove the law by induction over the degree of
$a$. Since $\deg(a)>n$, $a$ is not a generator. Let $h\in \g$ and $a=h.a'\in M$. We have
\begin{align*}
\phi(h.a').y&=(h\phi(a')).y=h.\phi(a').y-(-1)^{a'h}\phi(a').h.y=(\text{induction
over }a)\\
&=-(-1)^{a'y}h.\phi(y).a'-(-1)^{a'h}\phi(a').h.y=(\text{induction
over }y)\\
&=-(-1)^{a'y}h.\phi(y).a'+(-1)^{a'y}\phi(h.y).a'\\
\phi(y).h.a'&=(\phi(y)h).a'+(-1)^{yh}h.\phi(y).a'=\\
&=-(-1)^{yh}(h\phi(y)).a'+(-1)^{yh}h.\phi(y).a'
\end{align*}
Hence, $\phi(h.a').y=-(-1)^{a'y+yh}\phi(y).h.a'$. 
\end{proof}

 The following Theorem is an application of Theorem \ref{modul} to a
situation which is one step towards the situation we have in this paper.
\begin{thm}\label{one} Let $\g$ be a graded Lie superalgebra, generated in degree 1 and with
relations in degree $\le r$. Let $L$ be a
graded $\g$-module Lie superalgebra generated in degree 1 and let $\phi:\ L\to\g$ be a Lie
algebra homomorphism (not necessarily a $\g$-module
homomorphism). Suppose there is an $n\ge r$ such that \hfill\break 
$\phi_{\ge n}:L_{\ge n}\to \g_{\ge n}$ is a
surjective $\g$-module map and an isomorphism in degrees $n\le j\le
2n-1$. Suppose also
\begin{align}
\label{mod}
\phi(x).y&=[x,y]\quad\text{for}\ \ x,y\in L
\end{align}
Then $\phi$ is an isomorphism as Lie algebras in all degrees $\ge n$.
\end{thm}
\begin{proof}
Put $M=L_{\ge n}$. Then $M$ is generated by $M_n$, since $L$ is generated by
$L_1$ as a Lie algebra and hence $L_{n+j}$ is generated from $L_n$ by
operating with $\phi(L_1)$, using the assumption \eqref{mod}. The
commutative law \eqref{modmod} follows from \eqref{mod}, since $L$ is
a Lie algebra. Hence the result follows by Theorem \ref{modul}.

\end{proof}

Here, finally is a theorem which is adapted to our situation.
\begin{thm}\label{iso}
Let $\g$ be a graded Lie superalgebra, generated in degree 1 and with
relations in degree $\le r$. For $i=1,\ldots,s$, let $L_i$ be a
graded $\g$-module Lie superalgebra generated in degree 1 and $\phi_i:L_i\to\g$ be a Lie
algebra homomorphism (not necessarily a $\g$-module
homomorphism). Suppose there is an $n\ge r$ such that
$\phi_1,\ldots,\phi_s$ induce a map $\phi:\ (L_1\oplus
L_2\oplus\ldots\oplus L_s)_{\ge n}\to \g_{\ge n}$ which is a
surjective $\g$-module map and an isomorphism in degrees $n\le j\le
2n-1$. Suppose also for all $i=1,\ldots,s$,
\begin{align}
\label{modi}
\phi_i(x).y&=[x,y]\quad\text{for}\ \ x,y\in L_i\\
\nonumber\\
\label{ort}
\phi_i((L_i)_{\ge n}).(L_j)_{\ge n}&=0\quad\text{for}\ \ i\ne j
\end{align}
Then $\phi$ is an isomorphism as Lie algebras in all degrees $\ge n$. 
\end{thm}
\begin{proof}
We cannot apply Theorem \ref{one} directly, but we can slightly modify
its proof to get a proof of Theorem \ref{iso}. Let $L=L_1\oplus\dots\oplus L_s$ and
$\phi=\phi_1+\ldots+\phi_s:L\to\g$. Suppose $x,y\in L_{\ge n}$,
$x=x_1+\ldots+x_s, y=y_1+\ldots+ y_s$. Then by \eqref{modi} and 
\eqref{ort}
\begin{align*}
\phi(x).y=\sum_{ij}\phi_i(x_i).y_j=\sum_i\phi_i(x_i).y_i=\sum_i[x_i,y_i]=[x,y]
\end{align*}
This proves the condition \eqref{mod} of Theorem \ref{one} in degree
$\ge n$, which is enough to get \eqref{modmod} of Theorem \ref{modul}. The condition
\eqref{mod} is also
used to prove that $M$ is generated by $M_n$, but here it is enough to
prove that for each $i=1,\ldots,s$, $(L_i)_{n+j}$ is generated from
$(L_i)_{n}$ by operating with $(L_i)_1$ and this follows from \eqref{modi}. 
\end{proof}

\vspace{10pt}
Now we are able to prove the statements given in section \ref{sec:irr}, in
a more precise form. Observe that $\bar\eta$ in this section differs
from the one defined in section \ref{sec:irr} in the sense that it has one  
less basis element in degree two. Thus, the series \eqref{eq26} is
obtained by multiplying the series below by $\frac{1}{1-t^2}$.

\begin{crl} \label{cor}We have that $(\bar\eta)_{\ge3}$ is isomorphic as a Lie
algebra to 
$$
\bigl(\F(C,D',G)/<C^2,[C,[D',G]]>\oplus\F(D,E)\oplus\F(B,F)\bigr)_{\ge3}
$$ 
and the Hilbert series of the enveloping algebra of $\bar\eta$ is
$$
\frac{1-t}{(1-2t)^2(1-3t+t^2)}
$$
\end{crl}
\begin{proof}
We apply Theorem \ref{iso} with $n=r=s=3$. The premisses are
fullfilled by Proposition
\ref{sum}, Proposition \ref{homo}, Lemma \ref{orto} and the fact that
the map $\phi$ is an isomorphism in degree 3,4,5. This can be checked
either by using the program {\bf liedim} or make computations with the
commutative law in these degrees (which consists of a finite number of
equalities, e.g., $bbf.D_1=d.BBF$, $(g(d+e)c).fF=(ff).GD'C$). 

It is easy to see that the ideal $<C^2,[C,[D',G]]>$ in the free
associative algebra on $C,D',G$ is a Gr\"obner basis. Hence the series
for the quotient is the same as the series of the monomial algebra
$$
k<C,D',G>/<C^2,CD'G>
$$
whose series may be computed, using some combinatorial reasoning, to
be $1/(1-3t+t^2)$. Since we have the infinite product representation
for the series of an enveloping algebra mentioned in section \ref{sec:irr}, the
enveloping algebra of
$\bigl(\F(C,D',G)/<C^2,[C,[D',G]]>\bigr)_{\ge3}$ has the series 
$\frac{(1-t^2)^5}{(1+t)^3(1-3t+t^2)}$. Likewise the series of the other two
parts of $(\bar\eta)_{\ge3}$ have the series
$\frac{(1-t^2)^3}{(1+t)^2(1-2t)}$. Multiplying these three series together
gives the series for the enveloping algebra of $(\bar\eta)_{\ge3}$. To
get the series for the enveloping algebra of $\bar\eta$ one multiplies
further with $(1+t)^6/(1-t^2)^{10}$ which gives the result.
\end{proof}

\subsection{The extension by the radical}\label{sec:ext}

In \cite[chapter XIII (Lie algebras)]{Ca-Ei}, the
definition of the cohomology of a Lie algebra $\bar \g$ with coefficients
in a $\bar \g$-module $M$  is
given and in chapter XIV, section 5 (Extensions of Lie algebras), loc.cit.
an interpretation
of $H^2(\bar \g,M)$ is given as the classification of all Lie algebra
extensions 
$0 \rightarrow M \rightarrow \g \rightarrow \bar \g \rightarrow 0$ where $M$
is an abelian
ideal in $\g$ and $\bar \g$ is the quotient. It  is proved that the
extensions are classified
by all "2-cocycles"   $ c: \bar \g \times  \bar \g  \rightarrow M$ satisfying the
cocycle conditions and that
the Lie algebra structure on $\g = \bar \g \oplus M$ is given by vector space
addition and multiplication by
$[(\bar g_1,m_1),(\bar g_2,m_2)] =([\bar g_1,\bar g_2],\bar g_1.m_2 +m_1.\bar
g_2+c(\bar g_1,\bar g_2))$
Different 2-cocycles give isomorphic extensions if and only if they differ
by a coboundary of
a 1-cochain.
In the case we are considering, 
$\g$ corresponds to $\eta$, $\bar \g$ corresponds to
$\bar \eta$,  $M$ is $Rad$ and  furthermore everything is graded.

Thus, in order to show that 
the radical in $\eta$ is
exactly two-dimensional in each degree $\ge3$, we form the  
graded vector space $Rad$ which is two-dimensional in each degree $\ge3$
(and one-dimensional in degree two) and make it to an
$\be$-module, which is easy. Then we define a function $\fe:\be\times\be\to Rad$.
To verify that $\fe$ is a
cocycle consists of checking a large number of cases. When this is
done 
one checks that the Lie algebra $\hat\eta=Rad\oplus\be$ is generated
in degree one and that the elements that correspond to
relations in $\eta$ are zero, which is easy. Now $\hat\eta$ is a
quotient of $\eta$ but also as a graded vector space $\eta$ is a
quotient of $\hat\eta$ and it follows that $\hat\eta$ and $\eta$ are
isomorphic.

We will use the same notation for elements in $\bar\eta$ as for
elements in $\eta$,
but we use capital letters instead. We will use Corollary \ref{cor} to
get a basis for $\be$ as a vector space. 
Put 
\begin{align*}
L_1&=\bigl(\F(C,D',G)/<C^2,[C,[D',G]]>\bigr)_{\ge3}\\
L_2&=\F(D,E)_{\ge3}\\
L_3&=\F(B,F)_{\ge3}
\end{align*}
Let the sequences of elements in $\be$,
$B_n,F_n,D_n,E_n$, be defined as (cf. Proposition \ref{bf})
$B_1=B,F_1=F,D_1=D,E_1=E$, $B_2=B^2, F_2=F^2, D_2=D^2, E_2=E^2$
and $B_n$ for $n\ge3$ is
obtained by successive multiplication to the left by $B,F$ alternatively, starting
with $F$, and similarly for the three other sequences. Also define the
sequences $C_n,G_n$ for $n\ge1$ by, $C_n=(D')^{n-1}C,\
G_n=(D')^{n-1}G$. 
\begin{dfn} The sequences of elements $B_n,F_n,D_n,E_n,C_n,G_n$ for
$n\ge2$ are called
$\fe$-threads. The set $I$ is the linear span of all iterated Lie
products of degree $\ge3$ of the generators $B,C,D,E,F,G$ 
which are not $\fe$-threads.   
\end{dfn}
\begin{lm}\label{ideal}
$$\begin{array}{rl}
(i)&\quad 
I \text{
is an ideal  
in $\be$.}\\
(ii)&\quad 
[\be_{\ge3},\be_{\ge3}]\subset I\\
(iii)&\quad 
\text{The $\fe$-threads are linearly independent modulo $I$.}\\
(iv)&\quad  
F.D_n=(-1)^nE_{n+1}\text{ mod }I,\quad n\ge1,\quad F.E_n\in I,\quad n>1\\
(v)&\quad 
B.E_n=(-1)^nD_{n+1}\text{ mod }I,\quad B.D_n\in I,\quad n\ge1\\
(vi)&\quad D,E,C,G\quad\text{annihilate \ }B_n,F_n \text{ \ for
all}\quad n\ge3\\
(vii)&\quad B.C_n=B.G_n=0,\quad F.C_n,F.G_n\in I,\quad n\ge1\\
\end{array}
$$
\end{lm}
\begin{proof}
Let $I_D,I_E,I_B,I_F$ be the ideals in $L_2$
and $L_3$ respectively defined in
the proof of Proposition \ref{zero} such that in degree $n\ge3$, $L_2/I_D$
has basis $D_n$, $L_2/I_E$
has basis $E_n$, $L_3/I_B$
has basis $B_n$, $L_3/I_F$
has basis $F_n$. Let $I_C$ be the ideal in $L_1$ consisting of all
elements of degree $\ge3$ and of triple degree not of the form
$(1,n,0)$, where the triple degree of $C,D',G$ is
$(1,0,0),(0,1,0),(0,0,1)$ respectively. This is an ideal, since there
are no elements of triple degree $(0,n,0)$ for $n\ge3$. Similarly
$I_G$ is defined as all elements of degree $\ge3$ of triple degree not of the form
$(0,n,1)$. We have 
$$
I=(I_C\cap I_G)\oplus(I_D\cap
I_E)\oplus(I_B\cap I_F)
$$
which proves (i). 

We claim
that $L_3/I_B$ is abelian. It is bigraded (with $\bideg(B)=(1,0)$
and $\bideg(F)=(0,1)$) and $\bideg(B_n)=(n+2,n)$. The product
$[B_n,B_m]$ of
two basis elements have not the right bidegree, hence it must be
zero and the claim follows. Hence
$[L_3,L_3]\subset I_B$.  The same proof works for
$I_F$ and hence $[L_3,L_3]\subset I_B\cap I_F$. Analogously we get
$[L_2,L_2]\subset I_D\cap I_E$. The threads
$C_n$, $G_n$ are bases for $L_1/I_C$, $L_1/I_G$ in degree $n\ge3$. The
same argument as above proves $[L_1,L_1]\subset I_C\cap I_G$. Since
$\be_{\ge3}=L_1\oplus L_2\oplus L_3$ (ii) follows.

By Proposition \ref{zero}, For $n=2$ it follows from the basis for
$\be_2$. Suppose $n\ge3$.
$\{B_n,F_n\}$ is linearly independent 
modulo $I_B\cap I_F$ and the same is true for $\{D_n,E_n\}$. It is
also true 
for $\{C_n,G_n\}$ since $C_n$ and $G_n$ have different triple
degrees. Hence (iii) follows. 

(iv) and (v) are easily proven by induction, using the action of $B,F$
on $L_2$ and the definition of $B_n,F_n$.

(vi) and (vii) follows from the action of $D,E$ on $L_3$.

\end{proof}

Below, the definition of a function $\fe:\bar\eta\times\bar\eta\to
Rad$ is given, which factors through $\fe:\bar\eta/I\times\bar\eta/I\to
Rad$, where the $\bar\eta$-module $Rad$ has basis $R_2$ in degree 2 and basis $\{R_n,R_n'\}$ in
degree $n\ge3$ and $D.R_{2n}=R_{2n+1},\ D.R_{2n+1}'=R_{2n+2}',\
E.R_{2n+1}=R_{2n+2},\ E.R_{2n+2}'=R_{2n+3} $ for $n\ge1$ and all other
operations are defined as zero. It follows that
$(\be_{\ge3})Rad=0$ since $Rad$ is annihilated by $D^2,E^2$ and hence
by $L_2$. By Lemma \ref{fec} a basis for $\be/I$ is $B,C,D,E,F,G,B^2,F^2,D^2,E^2,DE,BF,DC,EG,G^2$ and $B_n,F_n,D_n,E_n,C_n,G_n$ for
all $n\ge3$. We define $\fe(x,y)=0$ if $x$ or $y$ is in $I$.

\vspace{10pt}
\begin{dfn} Let $m,n\ge1$
$$
\begin{array}{rclrcl}
\fe(D_m,F_{2n})&=&-R_{m+2n}' &\fe(D_m,B_{2n-1})&=&-R_{m+2n-1}\\
\fe(E_m,F_{2n-1})&=&R_{m+2n-1}'\quad \text{for }\ m\cdot n>1 &\fe(E_m,B_{2n})&=&R_{m+2n}\\
\fe(C_m,C_n)&=&\lambda_{m,n}R_{m+n}&\fe(E,CG)&=&-R_3'\\
\fe(C,G_n)&=&(-1)^nR_{n+1}'\quad \text{for }\ n>1&\fe(DE,CG)&=&-R_4'
\end{array}
$$
\end{dfn}
The coefficients $\lambda_{m,n}$ will be determined below. We have
$\lambda_{1,1}=2,\ \lambda_{2,1}=-\lambda_{1,2}=1$.
Except for symmetry, $\fe$
is defined as zero for all other pairs of basis elements of
$\be/I$. Observe that the definition concerning $D_m,E_m,B_n,F_n$
above for $m=1$ is in accordance with the result in Proposition \ref{bf}.

The
symmetric law and the cocycle condition reads:
\begin{eqnarray}
\nonumber
\fe(x,y)&=&-(-1)^{xy}\fe(y,x)\\
\label{coc} 
\fe([x,y],z)&=&\fe(x,[y,z])+(-1)^{yz}\fe([x,z],y)+\\
\nonumber
&&x\fe(y,z)-\fe(x,y)z+(-1)^{yz}\fe(x,z)y\\
\label{square}
\fe(x^2,y)&=&\fe(x,[x,y])+x\fe(x,y)-\frac{1}{2}\fe(x,x)y\quad \text{if
$x$ is odd}
\end{eqnarray}
where we have used that a left module over a Lie algebra also may be
considered as a right module by $mx=-(-1)^{xm}xm$. The equation \eqref{square} 
follows from \eqref{coc}.

These law impose some 
restrictions on the $\lambda_{m,n}$ and we will now prove  that there
is a  sequence $\lambda_{m,n}$
satisfying these restrictions and $\lambda_{1,1}=2$. 
\begin{lm}\label{lambda} There is a sequence of numbers $\lambda_{m,n}$
satisfying $\lambda_{1,1}=2$, $\lambda_{2,1}=1$ and the following conditions for all $m,n\ge1$.
\begin{align}
\label{even}
\lambda_{2m,2n}&=0\\
\label{symm}
\lambda_{m,n}&=-(-1)^{mn}\lambda_{n,m}\\
\label{rek}
\lambda_{m+1,n}&=\lambda_{m,n}+(-1)^{m+1}\lambda_{m,n+1}
\end{align}
\end{lm}
\begin{proof} Suppose the sequence is defined for all $m,n$ such that
$m+n<k$ and such that the conditions are fullfilled. Suppose first $k$ is even. We define
$\lambda_{1,k-1}=\lambda_{k-2,1}$ and for $m+n=k$, $m>1$ we define
$\lambda_{m,n}=\lambda_{m-1,n}$ for $m,n$ odd, and $\lambda_{m,n}=0$
for $m,n$ even. Then \eqref{even} holds and \eqref{rek} holds for
$m+n=k$ and $m,n$ odd. If $m,n$ are even, we have to prove that 
$0=\lambda_{m-1,n}+\lambda_{m-1,n+1}$. But
\begin{align*}
\lambda_{m-1,n}&=-\lambda_{m-2,n+1}\quad \text{by induction}\\
\lambda_{m-1,n+1}&=\lambda_{m-2,n+1}\quad \text{by definition}
\end{align*}

Also \eqref{symm} is true when $m=1,n=k-1$. We prove by
induction over $m$ that \eqref{symm} is true for all $m,n$ such that
$m+n=k$. We may suppose $m,n$ are odd and $>1$.
\begin{align*}
\lambda_{m,n}=\lambda_{m-1,n}=-\lambda_{n,m-1}=\lambda_{n-1,m}=\lambda_{n,m}
\end{align*}
Now suppose $k=2r+1$ is odd. We use \eqref{rek} to define $\lambda_{m,n}$
for $m+n=k$ and $m>1$. By making substitutions, it is possible to express
$\lambda_{m,n}$ in terms of $\lambda_{i,j}$ where $i+j=k-1$ and
$\lambda_{1,k-1}$. In particular, we have
$$
\lambda_{k-1,1}=\sum_{j=1}^{r}(-1)^{j+1}\lambda_{k-2j,2j-1}-(-1)^r\lambda_{1,k-1}
$$
It is easy to see, using \eqref{symm} that the sum $S$ above is zero when
$r$ is even, so in this case we get
$\lambda_{k-1,1}=-\lambda_{1,k-1}$, which is the symmetry law. We then
define $\lambda_{1,k-1}$ arbitrarily. If $r$ is odd, we define 
$\lambda_{k-1,1}=-\lambda_{1,k-1}=\frac{1}{2}S$. We now prove \eqref{symm} by
induction over $m$.
\begin{align*}
\lambda_{m,n}&=\lambda_{m-1,n}+(-1)^m\lambda_{m-1,n+1}=
-(-1)^{mn+n}\lambda_{n,m-1}+(-1)^{mn+n}\lambda_{n+1,m-1}=\\
&(-1)^{mn+n+n+1}\lambda_{n,m}=-(-1)^{mn}\lambda_{n,m}
\end{align*}  
\end{proof}

\begin{prp}\label{fec} The function $\fe$ is a cocycle.
\end{prp} 
\begin{proof} The cocycle condition \eqref{coc} is
(super)symmetric. Hence it is enough to consider unordered triples of
$x,y,z$ where not two even elements are equal. We will go through all
combinations of basis elements in $\be/I$. The main case will be when
one of $x,y,z$ is a $\fe$-thread.  If one of $x,y,z$ belongs to $I$,
all terms of \eqref{coc} will be zero by definition and Lemma
\ref{ideal} (i) and the fact that $(\be_{\ge3})Rad=0$. Thus, besides
the main case we only have to consider the case when $x,y,z$ all have
degree $\le2$ and none of them is a $\fe$-thread.

\vspace{10pt}
\noindent I. The  non-main case.

Suppose $x,y,z$ all have degree 1. It is easy to see that
the only cases when some term in \eqref{coc} is nonzero are
$\{B,B,E\}$, $\{F,F,D\}$, $\{D,D,B\}$, $\{E,E,F\}$, $\{C,C,D\}$ and
$\{E,C,G\}$.

\vspace{6pt}
\noindent
I.1 $x=y=B,z=E$.

We have to prove \eqref{square}, which gives
$$
\fe(B^2,E)=\fe(B,BE)+B\fe(B,E)-0=-\fe(B,D^2)=\fe(D^2,B)
$$

But by definition, $\fe(B^2,E)=-\fe(E,B^2)=-R_3$ and
$\fe(D^2,B)=-R_3$.

\vspace{6pt}
\noindent
I.2 $x=y=F,z=D$.

This gives
$$
\fe(F^2,D)=\fe(F,FD)+F\fe(F,D)-0=-\fe(F,E^2)=\fe(E^2,F)
$$

By definition, both sides are equal to $R_3'$.

\vspace{6pt}
\noindent
I.3 $x=y=D,z=B$.
$$
-R_3=\fe(D^2,B)=\fe(D,DB)+D\fe(D,B)-0=0-DR_2=-R_3
$$

\vspace{6pt}
\noindent
I.4 $x=y=E,z=F$.
$$
R_3'=\fe(E^2,F)=\fe(E,EF)+E\fe(E,F)-0=-\fe(E,CG)=R_3'
$$

\vspace{6pt}
\noindent
I.5 $x=y=C,z=D$.
\begin{align*}
0=&\fe(C^2,D)=\fe(C,CD)+C\fe(C,D)-\frac{1}{2}\fe(C,C)D=\\
&\lambda_{1,2}R_3-\frac{1}{2}\lambda_{1,1}R_2D=-R_3+R_3=0
\end{align*}

\vspace{6pt}
\noindent
I.6 $x=E,y=C,z=G$.
$$
0=\fe(EC,G)=\fe(E,CG)-\fe(EG,C)=-R_3'+R_3'=0
$$

\vspace{10pt}
Suppose $\deg(x)=\deg(y)=1$ and $\deg(z)=2$ and $z$ is not a
$\fe$-thread. We have the following cases to
consider: $\{B,D,BF\}$, $\{E,F,BF\}$, $\{C,C,DE\}$, $\{D,B,DE\}$,
$\{E,F,DE\}$, $\{C,G,DE\}$, $\{D,E,CG\}$.

\vspace{6pt}
\noindent
I.7 $x=B,y=D,z=BF$.
$$
0=\fe(BD,BF)=\fe(B,D_3)+\fe(-B_3,D)+0+0+0=0
$$
\vspace{6pt}
\noindent
I.8 $x=E,y=F,z=BF$.
$$
0=\fe(EF,BF)=\fe(E,-F_3)+\fe(E_3,F)+0+0+0=0
$$
\vspace{6pt}
\noindent
I.9 $x=C,y=C,z=DE$.
$$
0=\fe(C^2,DE)=\fe(C,-C_3)+0-\frac{1}{2}\fe(C,C)DE=(-\lambda_{1,3}+1)R_4
$$
which is true since, $\lambda_{1,3}=\lambda_{3,1}=\lambda_{2,1}=1$ by
Lemma \ref{lambda}.

\vspace{6pt}
\noindent
I.10 $x=B,y=D,z=DE$.
$$
0=\fe(BD,DE)=\fe(B,-D_3)+0-\fe(B,D)DE+0+0=R_4-R_4=0
$$

\vspace{6pt}
\noindent
I.11 $x=E,y=F,z=DE$.
$$
\fe(EF,DE)=0+\fe(F,-E_3)+0+0+0
$$
This follows since, $\fe(EF,DE)=-\fe(CG,DE)=-R_4'=-\fe(E_3,F)$.

\vspace{6pt}
\noindent
I.12 $x=C,y=G,z=DE$.
$$
R_4'=\fe(CG,DE)=\fe(C,-G_3)+0+0+0+0=R_4'
$$

\vspace{6pt}
\noindent
I.13 $x=D,y=E,z=CG$.
$$
-R_4'=\fe(DE,CG)=0+0+D\fe(E,CG)+0+0=-R_4'
$$

\vspace{10pt}
Suppose $\deg(x)=1$ or $\deg(x)=2$ and $\deg(y)=\deg(z)=2$ and $y,z$
are not
$\fe$-threads. Then $x$ and $y$ must be one of $DE,BF,CG,G^2$. But the
product of any two of these is in $I$ and hence all 
terms in \eqref{coc} are zero except for the case $x=D,E$, $y=DE$,
$z=CG$. But $D\fe(DE,CG)=0$, so $x=D$ does not give anything nonzero.

\vspace{6pt}
\noindent
I.14 $x=E,y=DE,z=CG$.
$$ 
0=\fe(EED,CG)=0+0+E\fe(DE,CG)-0+\fe(E,CG)DE=-R_5'+R_5'=0
$$

\vspace{10pt}
\noindent II. The main case.

By symmetry we may suppose $z$ is a $\fe$-thread. We will go through
all cases and find in each case the possible values of $x,y$ to get at
least one term nonzero in \eqref{coc}. 

\vspace{6pt}
\noindent
II.1 $z=C_m$.

By examining the multiplication tables for $\be$ one finds that the only cases to consider are when $x=D,E,DE$,
$y=C_n$, $z=C_m$ and then  
\eqref{coc} reads
$$
\fe(xy,z)=(-1)^{yz}\fe(xz,y)+x\fe(y,z)
$$

\vspace{6pt}
\noindent  II.1.1 $x=D$, $y=C_n$, $z=C_m$. 

If $n,m$ are even, all terms are zero by Lemma
\ref{threads} and \eqref{even}. If $n,m$ are both odd, we must prove
$$
\lambda_{n+1,m} =-\lambda_{m+1,n}+\lambda_{n,m}
$$
but this follows from \eqref{symm} and \eqref{rek}. If $n$ is even and
$m$ is odd, we are left with
$$
0=\lambda_{m+1,n}+0
$$
which is true by \eqref{even}. If $n$ is odd and $m$ is even, we are
left with $\lambda_{n+1,m}=0+0$ which again follows from \eqref{even}.

\vspace{6pt}
\noindent II.1.2 $x=E$, $y=C_n$, $z=C_m$. 
 
Suppose $n,m$ are even. Then we must prove
$$
\lambda_{n+1,m} =\lambda_{m+1,n}
$$
but this follows from \eqref{symm}, \eqref{even} and \eqref{rek}. If $n$ is even and
$m$ is odd, we are left with 
$$
\lambda_{n+1,m} =0+\lambda_{n,m}
$$
which is true by \eqref{even} and \eqref{rek}. If $n$ is odd and $m$
is even, we are left with
$$
0=\lambda_{m+1,n}+\lambda_{n,m}
$$
which follows from \eqref{symm}, \eqref{even} and \eqref{rek}. If
$n,m$ are odd, then all terms are zero.

\vspace{6pt}
\noindent II.1.3 $x=DE$, $y=C_n$, $z=C_m$. 

For all $n,m$ we get
$$
\lambda_{n+2,m}=(-1)^{nm}\lambda_{m+2,n}+\lambda_{n,m}
$$
This follows by applying \eqref{rek} twice and using \eqref{symm}.

\vspace{6pt}
\noindent
II.2. $z=G_m$.

The only cases to consider are when $x=D,E,DE$,
$y=C$, $z=G_m$ and in this case
\eqref{coc} reads
$$
0=(-1)^{yz}\fe(xz,y)+x\fe(y,z)
$$

\vspace{6pt}
\noindent
II.2.1 $x=D$, $y=C$, $z=G_m$.

This gives $0=0$ if $m$ is odd and when $m$ is even, we get
$$
0=\fe(G_{m+1},C)+D\fe(C,G_m)
$$
which is true by definition.

\vspace{6pt}
\noindent
II.2.2 $x=E$, $y=C$, $z=G_m$.

This gives $0=0$ if $m$ is even and when $m$ is odd we get
$$
0=-\fe(G_{m+1},C)+E\fe(C,G_m)
$$
which is true by definition.

\vspace{6pt}
\noindent
II.2.3 $x=DE$, $y=C$, $z=G_m$.

For all $m,n$ we get
$$
0=(-1)^m\fe(G_{m+2},C)+DE\fe(C,G_m)=-\fe(C,G_{m+2})+(-1)^mR_{m+3}'
$$
which is true by definition.

\vspace{6pt}
\noindent
II.3 $z=D_m$.

By Lemma \ref{ideal} and the multiplication table
$\be_1\times\be_2\to\be_3$ we get the following cases for $x,y$. One
is $F$ and the other is $B_{2n}$ or $F_{2n-1}$; one is $B_{2n-1}$ and
the other is $D,E,DE,BF$; one is $F_{2n}$ and the other is
$D,E,DE,BF$.

\vspace{6pt}
\noindent
II.3.1 $x=F$, $y=B_{2n}$, $z=D_m$.

We get
$$
\fe(B_{2n+1},D_m)=0+\fe((-1)^mE_{m+1},B_{2n})+0+0+0
$$
which is true by definition.

\vspace{6pt}
\noindent
II.3.2 $x=F$, $y=F_{2n-1}$, $z=D_m$.

Suppose $n>1$.
$$
\fe(F_{2n},D_m)=0+(-1)^m\fe((-1)^mE_{m+1},F_{2n-1})+0+0+0
$$
True by definition. Suppose $n=1$
$$
\fe(F^2,D_m)=\fe(F,(-1)^mE_{m+1})+0+0=\fe(E_{m+1},F)
$$
True by definition.

\vspace{6pt}
\noindent
II.3.3 $x=D$, $y=B_{2n-1}$, $z=D_m$.
$$
0=\fe(DB_{2n-1},D_m)=0+(-1)^m\fe(DD_m,B_{2n-1})+D\fe(B_{2n-1},D_m)
$$
If $m$ is even, then this is $0=0$. If $m$ is odd, then this gives
$R_{m+2n}-R_{m+2n}$.

\vspace{6pt}
\noindent
II.3.4 $x=E,DE$, $y=B_{2n-1}$, $z=D_m$.
 
Similar to II.3.3

\vspace{6pt}
\noindent
II.3.5 $BF$, $y=B_{2n-1}$, $z=D_m$.
$$
\fe((BF)B_{2n-1},D_m)=(-1)^m\fe(BFD_m,B_{2n-1})+0+0+0
$$
True since, $BFB_{2n-1}=B_{2n+1}$ and $BFD_{m}=-D_{m+2}$.

\vspace{6pt}
\noindent
II.3.6 $x=D$, $y=F_{2n}$, $z=D_m$.
$$
0=\fe(DF_{2n},D_m)=0+\fe(DD_m,F_{2n})+D\fe(F_{2n},D_m)+0+0
$$
If $m$ is even, both terms to the right are zero. If $m$ is odd, we
get $-R_{m+2n+1}'+R_{m+2n+1}'$.

\vspace{6pt}
\noindent
II.3.7 $x=E,DE$, $y=F_{2n}$, $z=D_m$.

Similar to II.3.6

\vspace{6pt}
\noindent
II.3.8 $x=BF$, $y=F_{2n}$, $z=D_m$.
$$
\fe((BF)F_{2n},D_m)=0+\fe((BF)D_m,F_{2n})+0+0+0=\fe(-D_{m+2},F_{2n})=R_{m+2n+2}'
$$
This is true since, $(BF)F_{2n}=F_{2n+2}$.

\vspace{6pt}
\noindent
II.4 $z=E_m$.

We get the following cases for $x,y$. One
is $B$ and the other is $F_{2n}$ or $B_{2n-1}$; one is $F_{2n-1}$ and
the other is $D,E,DE,BF$; one is $B_{2n}$ and the other is
$D,E,DE,BF$. The analysis in these cases are analogue to case II.3 (one
could also use the symmetry in the definition of $\fe$; the equalities
not concerning $C,G$ stay the same in degree $>2$ if $B$ and $F$ interchange and also
$D$ and $E$ and $R$ with $-R'$).

\vspace{6pt}
\noindent
II.5 $z=B_m$.

The possibilities in this case are: One is $D_n$ or $E_n$ the other is 
$D,E,B,F,DE,BF$. These cases have already been taken care of above.

\vspace{6pt}
\noindent
II.6 $z=F_m$.

The same possibilities as for $z=B_m$ and again these cases are
already checked.

\end{proof}

The fact that $\hat\eta=Rad\oplus\be$ is generated in degree one
follows easily since $Rad$ is generated by $R_2,R_3'$ and
$[C,C]_{\hat\eta}=\fe(C,C)=2R_2$ and $[C,[E,G]]_{\hat\eta}=R_3'$. 

Since $\fe(x,y)\ne0$ for $x,y$ of degree one only if $x=y=C$ or
$x=B,y=D$ it follows that all but one relation in $\eta$ correspond to
a relation in $\hat\eta$ since they hold in $\be$. The missing
relation is $BD+C^2$. This is valid since 
$$
[B,D]_{\hat\eta}+\frac{1}{2}[C,C]_{\hat\eta}=\fe(B,D)+\frac{1}{2}\fe(C,C)+[B,D]_{\be}+\frac{1}{2}[C,C]_{\be}=-R_2+R_2+0+0=0
$$
As was remarked in the beginning of this section, it follows that
$\hat\eta$ and $\eta$ are isomorphic since $\hat\eta$ is a quotient of
$\eta$ and on the same time, the dimension of $\hat\eta$ is at least
as big as the dimension of $\eta$  in each degree.

Thus the series for the enveloping algebra of $\eta$ will be a product of the series for the enveloping algebra of $\bar\eta$ and the infinite product 
$$
\frac{1}{1-t^2}\prod_{n=2}^{\infty}{(1+t^{2n-1})^2\over(1-t^{2n})^2}.
$$ 
Thereby, by Corolllary \ref{cor}, Theorem \ref{thm1} is
completely proved.

\section{Final remarks}
We have proved that the ring $R_{197}$ is a Gorenstein numerical semigroup ring
having an explicit transcendental Poincar\'e-Betti series, and we have said that the
same result remains true for the rings $R_n$ for $n=199,201,203,\ldots$.
Let us just illustrate this by comparing $R_{197}$ and $R_{199}$.
In both cases the rings $R_{197}/(a)$ and $R_{199}/(a)$ (where $a=t^{36}$) are ``the same''
but the gradings of the remaining generators $b,c,d,e,f,g,h,i,j,k,l$ are different:
Indeed the gradings are $(48,50,52,56,60,66,67,107,121,129,135)$ for $R_{197}/(a)$
and $(48,50,52,56,60,66,69,109,123,131,137)$ for $R_{199}/(a)$.
But we see in (\ref{eq5}) that if we choose the gradings $c_1=48,c_2=52$
we obtain that the gradings of $c=50,d=52,e=56,f=60,g=66,h=c_3,i=c_3+40,j=c_3+54,k=c_3+86,l=c3+68$
Thus if $h=c_3=67$, then $(i,j,k,l)=(107,121,129,135)$ i.e. the case $R_{197}/(a)$ which
we have studied above. But if $h=c_3=69$, then $(i,j,k,l)=(109,123,131,137)$, i.e. we are in case
$R_{199}/(a)$. This shows together with (\ref{eq5}) that the cases $R_{197}/(a)$ and $R_{199}/(a)$
can be regraded with $(b,c,d,e,f,g,h,i,j,k,l)=(1,1,1,1,1,1,1,2,2,2,2)$ so that they become isomorphic.
However the graded rings $R_{197}$ and $R_{199}$ can {\it not} be regraded so that they become isomorphic. 
Example: If we look for the possible  gradings of the the ideal J of (\ref{eq3})
we should as earlier solve the more complicated equations, where as above the relation $b^2-af$ is interpreted
as $2b-a-f=0$ etc.:

\begin{equation*}
\begin{aligned}[c]
2b-a-f& = 0\cr
c+c-b-d& = 0\cr
c+d-a-g& = 0\cr
d+d-b-e& = 0\cr
a+a+a-b-f& = 0\cr
d+e-b-f& = 0\cr
e+e-d-f& = 0\cr
e+f-c-g& = 0\cr
a+a+b-f-f& = 0\cr
a+a+c-e-g& = 0\cr
a+a+f-g-g& = 0\cr
a+b+c-h-h& = 0\cr
a+d+h-b-i& = 0\cr
\end{aligned}
\begin{aligned}[c]
c+i-a-j& = 0 \cr
a+e+h-d-i& = 0\cr
a+f+h-e-i& = 0\cr
b+c+h-a-k& = 0\cr
b+d+h-f-i& = 0\cr
a+g+h-b-j& = 0\cr
b+e+h-a-l& = 0\cr
c+j-a-l& = 0\cr
c+e+h-d-j& = 0\cr
g+i-d-j& = 0\cr
b+f+g-h-i& = 0\cr
c+f+h-b-k& = 0\cr
e+j-b-k& = 0\cr
\end{aligned}
\begin{aligned}[c]
d+f+h-c-k& = 0\cr
a+a+i-c-k& = 0\cr
b+g+h-d-k& = 0\cr
f+j-d-k& = 0\cr
c+g+h-b-l& = 0\cr
d+g+h-c-l& = 0\cr
e+k-c-l& = 0\cr
f+f+h-d-l& = 0\cr
g+j-d-l& = 0\cr
e+g+g-h-j& = 0\cr
e+g+h-f-k& = 0\cr
a+b+i-e-l& = 0\cr
f+g+h-a-a-j& = 0\cr
\end{aligned}
\begin{aligned}[c]
a+d+i-f-l& = 0\cr
g+k-f-l& = 0\cr
a+b+d+f-h-k& = 0\cr
a+a+k-g-l& = 0\cr
a+b+d+g-h-l& = 0\cr
a+d+f+g-i-i& = 0\cr
a+f+g+g-i-j& = 0\cr
b+h+j-i-k& = 0\cr
b+f+h+h-i-l& = 0\cr
j+j-i-l& = 0\cr
j+k-b-h-l& = 0\cr
f+h+k-j-l& = 0\cr
k+k-e-h-l& = 0\cr
a+i+j-k-l& = 0\cr
e+i+i-l-l& = 0\cr
\end{aligned}
\end{equation*}

The solutions of these $54$ equations in $12$ unknowns are:
$$
a=36c_1/67,b=48c_1/67,c=50c_1/67,d=52c_1/67,e=56c_1/67,f=60c_1/67,
$$
$$
g=66c_1/67,h=c_1,i=107c_1/67,j=121c_1/67,k=129c_1/67,l=135c_1/67
$$
where $c_1$ is a constant and since $67$ is a prime number, the minimal possible integral choice is $c_1=67$,
and so we find the unique grading of $R_{197}$.
But if we do the same reasoning for $R_{199}$ we still get 
$54$ relations between the $a,b,c,d,e,f,g,h,i,j,k$ corresponding to, but different from (\ref{eq3}):
$$
b^2-af, c^2-bd,cd - ag,d^2-be,de - bf,a^3-bf,e^2  - df, ef - cg, a^2b-f^2, a^2c - eg, a^2f -g^2 , a^2g -h^2,
$$
$$
adh - bi, ci - aj, aeh - di, afh - ei, bch - ak, bdh - fi, agh - bj, cj - al, beh - al, gi - dj,ceh-dj,dfg-hi,
$$
$$
 ej - bk, cfh - bk, a^2i - ck, dfh - ck, fj - dk, bgh - dk, cgh - bl, ek - cl, dgh-cl,gj - dl, f^2h - dl, egh - fk,
$$
$$
fg^2  - hj, abi - el, fgh - a^2j, gk - fl, adi - fl, abcg - hk, a^2k - gl,abeg - hl, acg^2  - i^2 , bdg^2  - ij, bhj - ik, j^2  - il,
$$
\begin{equation} \label{eq29}
 bfh^2  - il, jk - bhl, fhk - jl, k^2  - ehl, aij - kl,ei^2  - l^2 
 \end{equation}
and, when we want to analyze the possible gradings of $R_{199}$
we should now solve the 54 linear equations for the $a,b,c,d,e,f,g,h,i,j,k$  corresponding to (\ref{eq29}).
The solutions to these equations are now:
$$
a=12c_1/23,b=16c_1/23,c=50c_1/69,d=52c_1/69,e=56c_1/69,f=20c_1/23,
$$
\begin{equation} \label{eq30}
g=22c_1/23,h=c_1,i=109c_1/69,j=41c_1/23,k=131c_1/69,l=137c_1/69 
\end{equation}
But here $23$ is a prime number and the minimum possible choce of $c_1$ so that
all the gradings in (\ref{eq30}) become integers are $c_1=3\cdot 23=69$ and this gives again the original grading.

\end{document}